\documentclass[11pt,reqno]{amsart}
\usepackage[a4paper, total={6in, 8in}]{geometry}

\usepackage{enumerate}
\usepackage{amsmath}
\usepackage{amssymb,latexsym}
\usepackage{amsthm}
\usepackage{color}
\usepackage{cancel}
\usepackage{graphicx}

\usepackage{hyperref}

\newtheorem{theorem}{Theorem}[section]

\newtheorem{lemma}[theorem]{Lemma}
\newtheorem{corollary}[theorem]{Corollary}
\newtheorem{definition}[theorem]{Definition}
\newtheorem{proposition}[theorem]{Proposition}

\newtheorem{example}[theorem]{Example}

\newtheorem{remark}[theorem]{Remark}

\DeclareMathOperator{\GammaL}{\Gamma\mathrm{L}}

\newcommand{\fqn}{\mathbb{F}_{q^n}}
\newcommand{\fq}{\mathbb{F}_{q}}
\newcommand{\fqk}{\mathbb{F}_{q^k}}

\newcommand{\F}{{\mathbb F}}

\newcommand{\lmb}{\lambda}

\renewcommand{\mod}{\hbox{{\rm mod}\,}}

\newcommand{\N}{\mathrm{N}}

\begin{document}
\title{On generalized Sidon spaces}

\author[C. Castello]{Chiara Castello}
\address{Chiara Castello, \textnormal{Dipartimento di Matematica e Fisica, Universit\`a degli Studi della Campania ``Luigi Vanvitelli'', Viale Lincoln, 5, I--\,81100 Caserta, Italy}}
\email{chiara.castello@unicampania.it}

\maketitle   

\begin{abstract}
Sidon spaces have been introduced by Bachoc, Serra and Zémor as the $q$-analogue of Sidon sets, classical combinatorial objects introduced by Simon Szidon. In 2018 Roth, Raviv and Tamo introduced the notion of $r$-Sidon spaces, as an extension of Sidon spaces, which may be seen as the $q$-analogue of $B_r$-sets, a generalization of classical Sidon sets. Thanks to their work, the interest on Sidon spaces has increased quickly because of their connection with cyclic subspace codes they pointed out. This class of codes turned out to be of interest since they can be used in random linear network coding.
In this work we focus on a particular class of them, the one-orbit cyclic subspace codes, through the investigation of some properties of Sidon spaces and $r$-Sidon spaces, providing some upper and lower bounds on the possible dimension of their \textit{r-span} and showing explicit constructions in the case in which the upper bound is achieved. Moreover, we provide further constructions of $r$-Sidon spaces, arising from algebraic and combinatorial objects, and we show examples of $B_r$-sets constructed by means of them.
\end{abstract}



\section{Introduction}
Sidon spaces have been originally introduced by Bachoc, Serra and Zémor in \cite{BSZ2015}, in relation with the linear analogue of Vosper's Theorem (see also \cite{BSZ2018,HouLeungXiang2002}) as the $q$-analogue of Sidon sets, classical combinatorial objects introduced by Simon Szidon.\\
An $\fq$-subspace $V$ of $\fqn$ is called a \textbf{Sidon space} if the products of any two nonzero elements of $V$ are uniquely determined, up to a multiplicative factor in $\fq$, i.e. for all nonzero $a,b,c,d \in V$ such that $ab=cd$ then 
\[\{a \F_q,b \F_q\}=\{c \F_q,d \F_q\},\]
where $e\fq =\{e\lambda \colon \lambda \in \fq\}$. This can be naturally seen as the $q$-analog of Sidon sets originally introduced by Erd\H{o}s in \cite{Erdos}. A subset $S$ of an abelian group $(G,+)$ is a Sidon set if the sums of any pairs of elements (possibly identical) are distinct, i.e. if for all $a,b,c,d\in S$ such that $a+b=c+d$ then
\[
\lbrace a, b\rbrace=\lbrace c, d\rbrace.
\]

The interest on Sidon spaces has increased quickly especially after the work of Roth, Raviv and Tamo because in \cite{Roth} they pointed out the connection of Sidon spaces with \textbf{cyclic subspace codes}.\\
Let $k$ be a non-negative integer with $k \leq n$, the set of all $k$-dimensional $\F_q$-subspaces of $\F_{q^n}$, viewed as $\F_{q}$-vector space, forms a \textbf{Grassmannian space} over $\F_q$, which is denoted by $\mathcal{G}_{q}(n,k)$. A \textbf{constant dimension subspace code} is a subset $C$ of $\mathcal{G}_{q}(n,k)$ endowed with the metric defined as follows \[d(U,V)=\dim_{\F_q}(U)+\dim_{\F_q}(V)-2\dim_{\F_q}(U \cap V),\]
where $U,V \in C$. This metric is also known as \textbf{subspace metric}.
As usual, we define the \textbf{minimum distance} of $C$ as
\[ d(C)=\min\{ d(U,V) \colon U,V \in C, U\ne V \}. \]
Subspace codes have been recently used for the error correction in random
network coding, see \cite{KoetterK}. 
The first class of subspace codes studied was the one introduced in \cite{Etzion}, which is known as \textbf{cyclic subspace codes}.
A subspace code $C \subseteq \mathcal{G}_q(n,k)$ is said to be \textbf{cyclic} if for every $\alpha \in \F_{q^n}^*$ and every $V \in C$ then $\alpha V \in C$. If we denote by
\[
\mathrm{Orb}(V)=\lbrace \alpha V\colon \alpha\in\fqn^*\rbrace.
\]
and $C$ coincides with $\mathrm{Orb}(V)$, for some subspace $V$ of $\fqn$, then we say that $C$ is a \textbf{one-orbit} cyclic subspace code and $V$ is said to be a \textbf{representative} of the orbit. In this case, we define the \textbf{dimension} of $C$ as the dimension of $V$, since all the subspaces in $C$ have the same dimension.\\
There are constructions of cyclic subspace codes for several sets of parameters (close to the optimality, in a sense that we will explain soon), see e.g. \cite{code2,code1,code3,SantZullo1,zhang2022constructions,
zhang2023constructions,zhang2023new,
zhang2023further,Zullo}.
Let $V$ be an $\fq$-subspace of $\fqn$ of dimension $k$. Then the maximum field of linearity of $V$ is the maximum  subfield $\mathbb{F}_{q^t}$ of $\fqn$ such that $V$ is closed under multiplication by a scalar in $\mathbb{F}_{q^t}$. In particular it can be proved that $|\mathrm{Orb}(V)|=\frac{q^n-1}{q^t-1}$ if and only if $\mathbb{F}_{q^t}$ is the maximum field of linearity of $V$; see \cite[Theorem 1]{Otal}.
Therefore, every orbit of a subspace $V \in \mathcal{G}_q(n,k)$ defines a cyclic subspace code of size $(q^n-1)/(q^t-1)$, for some $t \mid n $. Let $V$ be a strictly $\fq$-linear subspace of dimension $k$ of $\fqn$. Then the cyclic subspace code defined by $V$ has size $(q^n-1)/(q-1)$. Also, in such a case, the maximum value for the minimum distance is at most $2k$ and it is exactly $2k$ if and only if the orbit of $V$ is a $k$-spread of $\F_{q^n}$, which happens if and only if $k=1$, since $V$ is strictly $\fq$-linear. So, when $k>1$

\[ d(\mathrm{Orb}(V))\leq 2k-2. \]

In \cite{Trautmann} the authors conjectured the existence of a cyclic subspace code of size $\frac{q^n-1}{q-1}$ in $\mathcal{G}_q(n,k)$ and minimum distance $2k-2$ for every positive integers $n,k$ such that $1<k\leq n/2$.

In \cite{BEGR} the authors used subspace polynomials to construct cyclic subspace codes with size $\frac{q^n-1}{q-1}$ and minimum distance $2k-2$, proving that the conjecture is true for any given $k$ and infinitely many values of $n$ and $q$. Such result was then improved in \cite{Otal}. 
Finally, the conjecture was solved in \cite{Roth} for most of the cases, by making use of Sidon spaces.

Roth, Raviv and Tamo in \cite{Roth} pointed out that the study of cyclic subspace codes with size $\frac{q^n-1}{q-1}$ and minimum distance $2k-2$ is equivalent to the study of Sidon spaces. 

\begin{theorem}\cite[Lemma 34]{Roth}\label{lem:charSidon}
Let $U \in \mathcal{G}_q(n,k)$.
Then $U$ is a Sidon space if and only if  $|\mathrm{Orb}(U)|=\frac{q^n-1}{q-1}$ and $d(\mathrm{Orb}(U))=2k-2$, i.e. for all $\alpha \in \mathbb{F}_{q^n}\setminus \fq$
\[\dim_{\fq}(U\cap \alpha U)\leq 1.\]
\end{theorem}

Moreover, the condition $k\leq n/2$ in the aforementioned conjecture may be obtained as a consequence of the following theorem. To this aim we need the following notation: if $V$ is an $\fq$-subspace of $\fqn$ then
\[ V^2=\langle ab \colon a,b \in V\rangle_{\fq} \]
and $V^2$ is said to be the \textbf{square-span} of V.

\begin{theorem}\cite[Theorem 18]{BSZ2015} \label{lowerboundSidon}
Let $V\in\mathcal{G}_q(n,k)$ be a Sidon space of dimension $k\geqslant 3$, then
\[ 
\dim_{\fq}(V^2)\geqslant 2k.
\]
\end{theorem}
Also, Sidon spaces with the smallest dimension of its square span were used in designing a multivariate public-key cryptosystem (see \cite{CryptoSidon}). This motivates the study of more general bounds on the dimension of the subspace generated by the $r$-products of $V$ (the $r$-span of $V$).\\
In \cite{Roth}, the notion of Sidon space was extended to the notion of $r$-Sidon space (see Section \ref{sec:r-sidon}), similarly to what has been done for the classical Sidon sets.
Indeed, this notion can be seen as the $q$-analog of the classical notion of $B_r$-sets defined in abelian groups; see e.g. \cite{GenSS}.
This property may be used to distinguish cyclic subspace codes, in the spirit of what happens for the square-span and more in general $r$-span of a subspace; see \cite{CPZ202x}, and also its investigation was strongly encouraged by Raviv, Langton and Tamo in \cite{CryptoSidon} in order to extend the hardness of the problem on which their cryptosystem is based.

In this paper we exploit the theory of $r$-Sidon spaces, first by giving some bounds on the dimension of an $r$-Sidon space obtained by evaluating the possible dimension of the $r$-span of an $r$-Sidon space (cf. Section \ref{sec:bounds}). In particular, we prove a lower bound on the dimension of $V^r$ for any $\fq$-subspace $V$ of $\fqn$ such that $1\in V$, which involves the stabilizer of $V^r$, i.e. the subfield of $\fqn$ containing the elements of $\fqn$ that fix $V^r$. Since in the prime degree extensions, the stabilizer of $V^r$ is forced to be $\fq$, except if the stabilizer coincided with $\fqn$, the lower bound on the dimension of $V^r$ can be improved. Moreover, if $k=\dim_{\fq}(V)$, we observe that the dimension of $V^r$ can be at most $\binom{k+r-1}{r}$ and we prove that any $\fq$-subspace $V$ of $\fqn$ such that $\dim_{\fq}(V^r)=\binom{k+r-1}{r}$ is an $r$-Sidon space. This allows us, in Section \ref{sec:maxrsidon}, to show some explicit constructions of $r$-Sidon spaces that reach this upper bound, that we define as \textbf{max-span r-Sidon spaces}, in analogy with the definition given by Roth, Raviv and Tamo in \cite{Roth} in the case $r=2$. Then we also explore constructions of $r$-Sidon spaces arising from scattered polynomials (cf. Section \ref{sec:constr}) and others with the aid of MAGMA computations \cite{MAGMA} (cf. Section \ref{sec:computat}).
We conclude the paper with Section \ref{sec:Brsets} where we show some examples of $B_r$-sets that can be constructed from our examples of $r$-Sidon spaces.

\section{Preliminaries}\label{sec:r-sidon}\label{sec:recall}

In the Introduction we mentioned that the definition of Sidon space was originally introduced in \cite{BSZ2015}.
A generalization of this concept has been given in \cite{Roth}, as the $q$-analogue of $B_r$-sets, where an $r$-Sidon space is defined as follows. In this section, we resume what is known about $r$-Sidon spaces, some results and examples of them and we will recall the definition of equivalence for Sidon spaces arising from the equivalence of subspace codes.

\begin{definition}
Let $k,n,r \in \mathbb{N}$ with $k<n$ and $r\geq 2$. A subspace $V \in \mathcal{G}_{q}(n,k)$ is said to be an \textbf{$r$-Sidon space} if for every nonzero $a_1,\ldots,a_r,b_1,\ldots,b_r \in V$, if
\[ \prod_{i=1}^r a_i = \prod_{i=1}^r b_i, \]
then the multisets $\{\{ a_1\fq,\ldots,a_r\fq \}\}$ and $\{\{ b_1\fq,\ldots,b_r\fq \}\}$ coincide.
\end{definition}

Notice that in an $r$-Sidon space the product of any $r$ nonzero elements is uniquely determined, up to scalars over $\fq$. 
Clearly, the following holds.

\begin{lemma}\label{lem:boxprop}
If $V \in \mathcal{G}_{q}(n,k)$ is an $r$-Sidon space, then it is an $r'$-Sidon space for any $r'$ such that $2\leq r'\leq r$.
\end{lemma}

In particular, any $r$-Sidon space is a Sidon space as well.

In \cite[Lemma 41]{Roth}, the following bound on the dimension of an $r$-Sidon space $V$ of $\fqn$ was proved:
\[ \dim_{\fq}(V) < \frac{n}r+1 +\log_q(r). \]

Also in \cite[Constructions 42 and 45]{Roth}, the following two constructions of $r$-Sidon spaces were shown:

\begin{itemize}
    \item Let $q$ be a prime power and let $k$ and $r$ be two integers. Let $n=k(r+1)$ and let $\gamma\in \fqn$ be a root of an irreducible polynomial of degree $r+1$ over $\fqk$. Then $V=\{u+u^q\gamma \colon u \in \fqk\}$ is an $r$-Sidon space.
    \item Let $q$ be a prime power and let $k$ and $r$ be two integers. Let $n=kr$ and let $\gamma\in \fqn$ be a root of an irreducible polynomial $p(x)$ of degree $r$ over $\fqk$ such that $p(0)$ is not a $(q-1)$-th power. Then $V=\{u+u^q\gamma \colon u \in \fqk\}$ is an $r$-Sidon space.
\end{itemize}

As already explained in the Introduction, Sidon spaces and cyclic subspace codes with a certain minimum distance are equivalent objects. Therefore, it is quite natural to give a definition of equivalence for Sidon spaces arising from the equivalence of subspace codes; see \cite{CPSZSidon} and \cite{Zullo}. 
The study of the equivalence for subspace codes was initiated by Trautmann in \cite{Trautmann2} and the case of cyclic subspace codes has been investigated in \cite{Heideequiv} by Gluesing-Luerssen and Lehmann.
Thus, motivated by \cite[Definition 3.5 and Theorem 6.2(a)]{Heideequiv}, the authors of \cite{CPSZSidon} defined two cyclic subspace codes as \textbf{semilinearly equivalent} if there exists $\sigma \in \mathrm{Aut}(\fqn)$ such that 
\[ \mathrm{Orb}(U)=\mathrm{Orb}(V^{\sigma}), \]
where $V^{\sigma}=\{ v^\sigma \colon v \in V \}$ and this happens if and only if $U=\alpha V^{\sigma}$, for some $\alpha \in \fqn^*$.
As a consequence, they gave the following definition of equivalence among two Sidon spaces.

\begin{definition}
Let $U$ and $V$ be two $\fq$-subspaces of $\fqn$. Then we say that $U$ and $V$ are \textbf{semilinearly equivalent} if the associated codes $\mathrm{Orb}(U)$ and $\mathrm{Orb}(V)$ are semilinearly equivalent, that is there exist $\sigma \in \mathrm{Aut}(\fqn)$ and $\alpha \in \fqn^*$ such that $U=\alpha V^{\sigma}$.
In this case, we will also say that they are equivalent under the action of $(\alpha, \sigma)$.
\end{definition}

In the following theorem, we will show that the property of $V$ of being an $r$-Sidon space is invariant under semilinear equivalence.

\begin{proposition}
    \label{rsidoninvar}
Let $V$ be an $r$-Sidon space of $\fqn$ of dimension $k$. Then $\alpha V^{\sigma}$ is an $r$-Sidon space of $\fqn$ of dimension $k$, for any $\alpha\in\fqn^*$ and $\sigma\in \mathrm{Aut}(\fqn)$.
\end{proposition}
\begin{proof}
Clearly $\dim(\alpha V^{\sigma})=k$. Assume that 
\[ \prod_{i=1}^r (\alpha a_i^{\sigma}) = \prod_{i=1}^r (\alpha b_i^{\sigma}). \]
for some $a_1,\dots,a_r\in V$ and $b_1,\dots,b_r\in V$. Thus, since $\alpha\neq 0$ and $\sigma\in\mathrm{Aut}(\fqn)$ then
\begin{equation}
 \label{eq:prodr}
    \prod_{i=1}^r a_i = \prod_{i=1}^r b_i. 
 \end{equation}
Since $V$ is an $r$-Sidon space, Equation \eqref{eq:prodr} implies
\[
\{\{ a_1\fq,\ldots,a_r\fq \}\}=\{\{ b_1\fq,\ldots,b_r\fq \}\},
\]
and so
\[
\{\{ \alpha a_1^{\sigma}\fq,\ldots,\alpha a_r^{\sigma}\fq \}\}=\{\{\alpha  b_1^{\sigma}\fq,\ldots,\alpha b_r^{\sigma}\fq \}\}.
\]
\end{proof}

\section{Bounds on Sidon spaces and max-span $r$-Sidon spaces}\label{sec:bounds}

In this section we provide a bound on the dimension of a Sidon space $V$ in the case in which we have some information on the $s$-span of $V$ for some positive integer $s\leq 2$. Indeed, this bound follows from a lower bound on the dimension of the $s$-span of $V$ which involves the stabilizer of $V^s$. Then, we also provide an upper bound on the dimension of $V^s$, and we show how to use it in order to get constructions of $r$-Sidon spaces.\\
\\
To this aim, let start by recalling the following theorem, known as the linear analogue of Kneser's Theorem. Consider a finite field $\fqn$ and let $A,B$ be two non-zero $\fq$-subspaces of $\fqn$. Then we denote by
\[ AB=\langle ab \colon a\in A, b \in B\rangle_{\fq}. \]



For the finite field case, \cite[Theorem 2.4]{HouLeungXiang2002} reads as follows.

\begin{theorem}
\label{thm:analoguekneserthm}
Let $A,B$ be non-zero
$\fq$-subspaces of $\fqn$. Then
\[
\dim_{\fq}(AB)\geqslant \min \lbrace n, \dim_{\fq}(A)+\dim_{\fq}(B)-\dim_{\fq}(H(AB))\rbrace,
\]
where $H(AB)=\lbrace x\in \fqn^*\colon xAB= AB\rbrace \cup \{0\}$ is the stabilizer of $AB$ in $\fqn$.
\end{theorem}

Since $H(AB)$ is a subfield of $\fqn$, if $\fqn$ is a prime extension of $\fq$ and $H(AB)\neq \fqn$, then $H(AB)=\fq$ and so Theorem \ref{thm:analoguekneserthm} represents the linear analogue of the well-known \textbf{Cauchy-Davenport inequality}. The following result characterizes the subspaces attaining the equality in the Cauchy-Davenport inequality, which can be seen as the linear analogue of Vosper's Theorem.

\begin{theorem}\cite[Theorem 3]{BSZ2015}
    \label{vosper}
Let $\fqn$ be a prime extension of $\fq$. Let $S,T$ be $\fq$-subspaces of $\fqn$ such that $2\leq \dim_{\fq}(S),\dim_{\fq}(T)$ and $\dim_{\fq}(ST)\leq n-2$. If
\[
\dim_{\fq}(ST)=\dim_{\fq}(S)+\dim_{\fq}(T)-1,
\]
then there are bases of $S$ and $T$ respectively of the form
\begin{center}
$\lbrace g, ga, \dots, ga^{\dim_{\fq}(S)-1}\rbrace$ and $\lbrace g', g'a, \dots, g'a^{\dim_{\fq}(T)-1}\rbrace$
\end{center}
for some $g,g',a\in \fqn$.
\end{theorem}

\begin{remark}
\label{rmk:h(v)=h(gamma-1v)}
    Given an $\fq$-subspace $V$ of $\fqn$, we will frequently use that \[ \dim_{\fq}(V)=[H(V)\colon \fq]\dim_{H(V)}(V).\] Indeed, let notice that the stabilizer $H(V)$ of $V$ is a subfield of $\fqn$ containing $\fq$ and $V$ is linear over $H(V)$. Hence, if $1\in V$ then $H(V)\subseteq V$ and \[\dim_{\fq}(V)=[H(V)\colon \fq]\dim_{H(V)}(V).\]
    If $1\notin V$ then $1\in \gamma^{-1}V$ for some $\gamma\in V$, $\dim_{\fq}(V)=\dim_{\fq}(\gamma^{-1}V)$ and $H(V)=H(\gamma^{-1}V)$. Therefore, also in this case \[\dim_{\fq}(V)=\dim_{\fq}(\gamma^{-1}V)=[H(\gamma^{-1}V)\colon \fq]\dim_{H(\gamma^{-1}V)}(\gamma^{-1}V)=[H(V)\colon \fq]\dim_{H(V)}(V).\]
    We will denote by $h(V)$ the degree of the field extension $H(V)$, i.e. $h(V)=[H(V)\colon \fq]$.
\end{remark}

As an immediate consequence of Remark \eqref{rmk:h(v)=h(gamma-1v)}, the following lemma holds.

\begin{lemma}
\label{lem:hv=fq}
Let $U$ be an $\fq$-subspace of $\fqn$ if dimension $k$ such that $\gcd(k,n)=1$. Then $H(V)=\fq$.
\end{lemma}

The following bounds on the dimension of a Sidon space $V$ can be useful to restrict the possible dimension of a Sidon space when we have some information, not only on $V^2$, but also on $V^s$ for some $s>2$.

\begin{definition}
Let $V$ be an $\fq$-subspace of $\fqn$ and $s \in \mathbb{N}$. Then
\[
V^s=\langle v_1v_2\cdots v_s : v_i\in V, i\in [s]\rangle_{\fq}
\]
is said to be the \textbf{$s$-span} of V over $\fq$.
\end{definition}

Moreover, we provide a natural upper bound on the dimension of the $r$-span of an $\fq$-subspace of $\fqn$ proving also that the subspaces attaining the equality in such a bound are $r$-Sidon spaces.

\begin{lemma}
\label{lem:h(t-1)<h(t)}
Let $V$ be an $\fq$-subspace of $\fqn$. Then for any positive integer $s$ it holds:
\[ H(V^{s-1})\subseteq H(V^s) \text{ and }  h(V^{s-1})\leq h(V^s). \]
\end{lemma}
\begin{proof}
Let $\lmb\in H(V^{s-1})$ with $\lambda \ne 0$ then
\[
\lmb V^{s-1}=V^{s-1}.
\]
Since $V^s=V^{s-1}V$, then
\[
\lmb V^s=\lmb (V^{s-1}V)=(\lmb V^{s-1})V=V^{s-1}V=V^s,
\]
therefore $\lmb\in H(V^s)$.
\end{proof}

\begin{remark}
    \label{h(vs)leqdeg fq(V)}
Let notice that if $V$ is an $\fq$-subspace of $\fqn$ such that $1\in V$ then $H(V^s)\subseteq\fq(V)$. Indeed, let $s\in\mathbb{N}$ and $\alpha\in H(V^s)$ then 
\[
\alpha V^s\subseteq V^s.
\]
and, since $1\in V$ then $\alpha\in V^s$. Therefore $H(V^s)\subseteq V^s$ and, since $V^s\subseteq \fq(V)$ for any $s\in\mathbb{N}$, then $H(V^s)\subseteq \fq(V)$.  
\end{remark}

We can now prove a bound on the dimension of a Sidon space which involves the stabilizer of its $s$-span for some values of $s$.

\begin{theorem}
\label{thm:lowerbound}
Let $V \in \mathcal{G}_q(n,k)$ be a Sidon space containing $1$ of dimension $k\geq 3$. 
Then
\begin{enumerate}
    \item there exists $t \in \mathbb{N}$ such that $t=\min\lbrace l\in\mathbb{N}: V^l=\fq(V)\rbrace$;
    \item $\dim_{\fq}(V^s)\geqslant sk-(s-2)h(V^s)$ and $k\leq [\fq(V)\colon \fq] \left( 1-\frac{1}{s}\right)$ for any $2 \leq s \leq t$.
\end{enumerate}
In particular, 
\[
k\leq [\fq(V)\colon \fq] \left( 1-\frac{1}{t}\right).
\]
\end{theorem}
\begin{proof}
We first prove the existence of such a $t$. Let notice that since $1\in V$, by Remark \ref{h(vs)leqdeg fq(V)} we get
\[
V\subseteq V^2\subseteq V^3\subseteq \ldots\subseteq V^s\subseteq \ldots\subseteq \fq(V)
\]
and so this chain is finite and we can consider
\[ t= \min \{ s \in \mathbb{N} \colon V^s=V^{s+1} \}. \]
Let $\{a_1,\ldots,a_k\}$ be a basis of $V$, then by definition we have that
\[ V^{t+1}=a_1V^{t}+\ldots+a_kV^{t}, \]
and so \[a_1V^{t}=a_2V^{t}=\ldots=a_kV^{t}=V^{t},\]
implying that $V^{t}$ is $\fq(V)$-linear.
Since $V^t \subseteq \fq(V)$, this implies that $V^{t}=\fq(V)$, hence the first item is proved.\\
For the second item, let us proceed by induction on $s$. \\
For $s=2$ the statement corresponds to Theorem \ref{lowerboundSidon}.\\
Let $2< s\leq t$ and suppose the assertion true for $s'<s$. Thus, by induction we have that 
\[
\dim_{\fq}(V^{s-1})\geqslant (s-1)k-(s-3)h(V^{s-1}).
\]
Note that since $s\leq t$, then $\dim_{\fq}(V)+\dim_{\fq}(V^{s-1})-h(V^s)<n$. Therefore, by Theorem \ref{thm:analoguekneserthm} and Lemma \ref{lem:h(t-1)<h(t)} we have that
\begin{equation}
\begin{aligned}
\dim_{\fq}(V^s)&=\dim_{\fq}(V V^{s-1})\geqslant \dim_{\fq}(V)+\dim_{\fq}(V^{s-1})-h(V^s)\\
&\geqslant k+(s-1)k-(s-3)h(V^{s-1})-h(V^s)\\
&=sk-(s-2)h(V^s).
\end{aligned}
\end{equation}
Moreover, by Remark \ref{h(vs)leqdeg fq(V)} we have that $h(V^s)\leq [\fq(V):\fq]$ and since $1\in V$ we have $V^s\subseteq V^t=\fq(V)$ for any $s\leq t$, thus $\dim_{\fq}(V^s)\leq [\fq(V)\colon \fq]$. Then we get
\[
k\leq [\fq(V)\colon \fq]\left(1-\frac{1}{s}\right).
\]
\end{proof}

\begin{remark}
    Let notice that if $1\notin V$ it is not true in general that $H(V^s)\subseteq V^s$. Indeed, let $V=\gamma\fq$ where $\gamma$ is a generating element of $\fqn$ over $\fq$. Since $\fq\subseteq H(V)$ we can consider $\alpha\in\fq$ and clearly $\alpha\in H(V)$ but $\alpha\notin V^s$ since $1\notin V$. Therefore $H(V^s)\not\subseteq V^s$.\\ 
    However, the assumption that $1 \in V$ is not restrictive since $\alpha V$ is still a Sidon space for every $\alpha \in \fqn^*$ and the bound provided in the above theorem still holds by replacing $\fq(V)$ by $\fq(\alpha^{-1} V)$, where $\alpha \in V \setminus\{0\}$.
\end{remark}

Let notice that when $n$ is prime $H(V^s)=\fq$ whenever $V^s \ne \F_{q^n}$, then the above result can be improved by using the linear analogue of Vosper's Theorem, i.e. Theorem \ref{vosper} as shown in the following theorem

\begin{theorem}
\label{thm:lowerboundprime}
Let $n$ be a prime number and let $V \in \mathcal{G}_q(n,k)$ be a Sidon space containing $1$ of dimension $k\geq 3$ and let $t$ be the minimum integer such that $V^t=\fqn$. If $t>2$ then
\begin{equation}
\label{eq:lowboundprime}
   \dim_{\fq}(V^s)\geq sk 
\end{equation}
for any $2 \leq s \leq t-1$. In particular, 
\begin{equation}
\label{eq:lowbounddimprime}
k\leq \left\lfloor \frac{n}{t-1}\right\rfloor.  
\end{equation}
\end{theorem}
\begin{proof}
Let us proceed by induction on $s$. \\
For $s=2$ Equation \eqref{eq:lowboundprime} follows from Theorem \ref{lowerboundSidon}.\\
Let $s>2$ and $s\leq t-1$ and suppose Equation \eqref{eq:lowboundprime} holds for any $s'<s$. Then by induction we have that 
\[
\dim_{\fq}(V^{s-1})\geqslant (s-1)k.
\]
Since $n$ is prime $h(V^i)=1$ for any $i<t$. Thus, since $s\leq t-1$, $\dim_{\fq}(V^s)<n$ and by Theorem \ref{thm:analoguekneserthm} and Lemma \ref{lem:h(t-1)<h(t)} it follows that
\begin{equation}
\begin{aligned}
\dim_{\fq}(V^s)&=\dim_{\fq}(V V^{s-1})\geqslant \dim_{\fq}(V)+\dim_{\fq}(V^{s-1})-1\\
&\geqslant k+(s-1)k-1\\
&= sk-1
\end{aligned}
\end{equation}
Now we prove that it cannot happen that $\dim_{\fq}(V^s)=sk-1$. \\
If $\dim_{\fq}(V^s)=n-1$ then
\[
n-1=\dim_{\fq}(V^s)=sk-1
\]
therefore $s\mid n$ and this is not possible since $n$ is prime. \\
If $\dim_{\fq}(V^s)\leqslant n-2$ then Theorem \ref{vosper}, applied to $V$ and $V^{s-1}$, implies that $V$ has a basis in geometric progression, up to a multiplicative factor in $\fqn$, that is
\[
V=g\langle 1,a,\dots,a^{k-1}\rangle_{\fq},
\]
for some $a,g\in\fqn$.\\
By hypothesis $V$ is a Sidon space and by Theorem \ref{lem:charSidon} we have that $\dim_{\fq}(V\cap\alpha V)\leqslant 1$ for any $\alpha\in\fqn^*$.  Therefore, the only possibility for $V$ is that $\dim_{\fq}(V)=k=2$. Since $k\geqslant 3$, a contradiction arises. \\
Therefore
\[
\dim_{\fq}(V^s)\geqslant sk
\]
and so by choosing $s=t-1$, since $(t-1)k\leq \dim_{\fq}(V^{t-1})<n$, we get
\[
k\leq \left\lfloor \frac{n}{t-1}\right\rfloor.
\]
\end{proof}


Now we want to provide also an upper bound on the dimension of $V^r$. Let denote by $\fq[x_1,\dots,x_k]_r$ the $\fq$-vector space of multivariate homogeneous polynomials of degree $r$ over $\fq$ for any $r\in\mathbb{N}$ and note that this $\fq$-vector space has dimension $\binom{k+r-1}{r}$ over $\fq$. Then the following result holds.

\begin{proposition}
\label{prop:upperbound}
Let $V$ be an $\fq$-subspace of $\fqn$ and let $\dim_{\fq}(V)=k$. Then
\[
\dim_{\fq}(V^r)\leqslant \min\left\lbrace n,\binom{k+r-1}{r}\right\rbrace.
\]
\end{proposition}
\begin{proof}
Let observe that if $V=\langle a_1,a_2,\dots,a_k\rangle_{\fq}$ then a generator system for $V^r$ is given by 
\[
\Gamma=\lbrace a_{i_1}a_{i_2}\cdots a_{i_r}:i_1\leq i_2\leq\dots\leq i_r, i_j\in[k], j\in[r]\rbrace
\]
and so
\[
\dim_{\fq}(V^r)\leqslant |\Gamma |=\binom{k+r-1}{r}.
\]

\end{proof}

The previous lemma gives us an upper bound on the possible dimension of an $r$-Sidon space. The following theorem gives us a sufficient condition in order to have an $r$-Sidon space, obtained as a generalization of Lemma 20 in \cite{Roth}.\\

\begin{theorem}\label{thm:maxspanrsidonspace}
If $V\in\mathcal{G}_q(n,k)$ is such that $\dim_{\fq}(V^r)=\binom{k+r-1}{r}$, then V is an $r$-Sidon space.
\end{theorem}
\begin{proof}
 Let $\mathbf{a}=(a_1,\ldots,a_k) \in \mathbb{F}_{q^n}^k$ whose components form an $\mathbb{F}_q$-basis of $V$ and let $u_1,\dots,u_r,v_1,\dots,v_r\in V\setminus \lbrace 0\rbrace$.\\
Let denote by
\[ 
u_i=\displaystyle\sum_{s\in [k]}\alpha_{i,s}a_s=p_{u_i}(\mathbf{a})
\]
\[
v_j=\displaystyle\sum_{t\in [k]}\beta_{j,t}a_t=p_{v_j}(\mathbf{a})
\]
for $i,j\in [r]$, where $p_{u_i},p_{v_j}$, $i,j\in [r]$, are the following multivariate polynomials over $\mathbb{F}_q$ in the indeterminates $x_1,\ldots,x_k$:
\[
p_{u_i}(x_1,\ldots,x_k)=\displaystyle\sum_{s\in [k]}\alpha_{i,s}x_s
\]
\[
p_{v_j}(x_1,\ldots,x_k)=\displaystyle\sum_{t\in [k]}\beta_{j,t}x_t
\]
Let observe that $u_1u_2\cdots u_r=v_1v_2\cdots v_r$ implies
\begin{equation}
\label{papb=pcpd}
\begin{aligned}
p_{u_1}(\mathbf{a})p_{u_2}(\mathbf{a})\cdots p_{u_r}(\mathbf{a})&=\displaystyle\sum_{\substack {i_1\leqslant\dots\leqslant i_r,\\ 
i_1,\dots,i_r\in [k]}}\alpha_{i_1,i_2,\dots,i_r}a_{i_1}a_{i_2}\dots a_{i_r}\\
&=\displaystyle\sum_{\substack {j_1\leqslant\dots\leqslant j_r,\\ 
j_1,\dots,j_r\in [k]}}\beta_{j_1,j_2,\dots,j_r}a_{j_1}a_{j_2}\dots a_{j_r}\\
&=p_{v_1}(\mathbf{a})p_{v_2}(\mathbf{a})\dots p_{v_r}(\mathbf{a}).
\end{aligned}
\end{equation}
Let observe that since $\dim_{\fq}(V^r)=\binom{k+r-1}{r}$ then both $\fq[x_1,\dots, x_k]_r$ and $V^r$ are $\mathbb{F}_q$-vector spaces of the same dimension. Once fixed the ordered basis $(x_{i_1}x_{i_2}\dots x_{i_r}: i_1\leqslant \dots\leqslant i_r, i_j\in [k], j\in [r])$ of $\fq[x_1,\dots, x_k]_r$ and the ordered basis $(a_{i_1}a_{i_2}\dots a_{i_r}: i_1\leqslant \dots\leqslant i_r, i_j\in [k], j\in [r])$ of $V^r$, there exists a unique isomorphism \[
\begin{aligned}
\phi\colon \fq[x_1,\dots,x_k]_r &\to V^r
\end{aligned}
\]
such that if 
\[
L(x_1,\ldots,x_k)=\displaystyle\sum_{\substack {i_1\leqslant\dots\leqslant i_r,\\ 
i_1,\dots,i_r\in [k]}}\alpha_{i_1,i_2,\dots,i_r}x_{i_1}x_{i_2}\dots x_{i_r}\in\fq[x_1,\dots,x_k]_r,
\]
then 
\[
\phi(L(x_1,\ldots,x_k))=L(a_1,\ldots,a_k)=\displaystyle\sum_{\substack {i_1\leqslant\dots\leqslant i_r,\\ 
i_1,\dots,i_r\in [k]}}\alpha_{i_1,i_2,\dots,i_r}a_{i_1}a_{i_2}\dots a_{i_r}\in V^r.
\]
Therefore, Equation \eqref{papb=pcpd} implies a polynomial identity, that is
\[
p_{u_1}(x)p_{u_2}(x)\cdots p_{u_r}(x)=p_{v_1}(x)p_{v_2}(x)\cdots p_{v_r}(x).
\]
Also, the polynomials $p_{u_1},p_{u_2},\ldots, p_{u_r}, p_{v_1},p_{v_2},\ldots, p_{v_r}$ are irreducible over $\mathbb{F}_q$ since they have degree 1 and $\mathbb{F}_q[x_1,x_2,\dots,x_n]$ is a unique factorization domain, then for every $i\in[r]$ there exists $j\in[r]$ and $\rho_{i,j} \in \fq$ such that 
\[
p_{u_i}(x)=\rho_{i,j} p_{v_j}(x).
\]
Without loss of generality we may suppose $p_{u_i}(x)=\rho_{i,i}p_{v_i}(x)=\rho_i p_{v_i}(x)$ for every $i\in [r]$,  then 
\[
p_{u_i}(x)=\displaystyle\sum_{s\in [k]}\alpha_{i,s}x_s=\displaystyle\sum_{s\in [k]}\rho_i\beta_{i,s}x_s=\rho_i p_{v_i}(x)
\]
and thus
\[
\alpha_{i,s}=\rho_i\beta_{i,s}
\]
Therefore, $u_i=\rho_i v_i$ for every $i\in [k]$ then $\lbrace\lbrace u_1\mathbb{F}_q,\dots,u_k\mathbb{F}_q\rbrace\rbrace =\lbrace\lbrace v_1\mathbb{F}_q,\dots,v_k\mathbb{F}_q\rbrace\rbrace$.
\end{proof}

Similarly to the Sidon space case, we define \textbf{max-span} $r$-Sidon spaces those $r$-Sidon spaces that satisfy the equality in Proposition \ref{prop:upperbound}. In Theorem \ref{thm:maxspanrsidonspace}, we have proved that every subspace that satisfies the equality in Proposition \ref{prop:upperbound} is an $r$-Sidon space and hence is a max-span $r$-Sidon space.

\section{Constructions of max span r-Sidon spaces} 
\label{sec:maxrsidon}

In the last result of the above section, we have seen that when $V^r$ has dimension $\binom{k+r-1}{r}$ it turns out that $V$ is an $r$-Sidon space.
In this section, we provide some explicit constructions of such spaces obtained by generalizing Construction 21 and Construction 23 in \cite{Roth}, by making use of classical combinatorial objects and with a more algebraic approach.
We start with the former approach by introducing the notion of $B_r$-set. 
A $B_r$-set in an abelian group can be defined as follows.

\begin{definition}
A subset $S$ of an abelian group $(G,+)$ is called a \emph{$B_r$-set} if the sums of the elements of all multi-sets of size $r$ of $S$ are distinct.
\end{definition}
In particular, for $r=2$ we get the definition of Sidon set.

\begin{remark}
\label{rmk:countbinom}  
Let remember that $\binom{h+r-1}{r}$ counts all the possible multi-subsets of order $r$ of a multi-set of elements of order $h$, such that every element can be repeated up to $r$ times in the same multi-subset, without caring about the order of the elements in each multi-subsets. Therefore, in such a case, two multi-subsets are distinct if and only if they differ for at least one element (taking into account the multiplicity).
\end{remark}

\begin{theorem}
\label{thm:maxspanbrset}
Let $\mathcal{S}:= \lbrace n_1,n_2,\dots , n_k\rbrace\subseteq [h]$ be a $B_r$-set in $\mathbb{Z}$, where $r\geq 2$, and let consider an integer $n>rh$ and an element $\gamma$ of $\mathbb{F}_{q^n}$ which does not belong to any proper subfield of $\mathbb{F}_{q^n}$. Let $V:= \langle \lbrace \gamma ^{n_i}\rbrace_{i\in [k]} \rangle_{\fq}$, then $V$ is a max-span $r$-Sidon space.
\end{theorem}
\begin{proof}
Let observe that $V:= \langle \lbrace \gamma ^{n_i}\rbrace_{i\in [k]} \rangle_{\fq}$ is an $\fq$-subspace of dimension $k$ over $\mathbb{F}_q$. Indeed, by hypothesis $\fq(\gamma)=\fqn$, hence the elements $1,\gamma,\dots,\gamma^{n-1}$ are $\fq$-linearly independent. Since $n>2h$, then $\gamma^{n_i}\neq \gamma^{n_j}$ for all $i,j\in [k], i\neq j$. Therefore, $\dim_{\fq}(V)=k$.\\
In order to prove that $V$ is a max-span Sidon space, by Theorem \ref{thm:maxspanrsidonspace} it is sufficient to show that $\dim_{\fq}(V^r)=\binom{k+r-1}{r}$.\\
Clearly, 
\[V^r=\langle \lbrace \gamma^{n_{i_1}+\dots +n_{i_r}}: i_j\in [k], j\in [r]\rbrace \rangle.\]
Also, since $\mathcal{S}=\lbrace n_1,n_2,\dots,n_k\rbrace$ is a $B_r$-set then all the sums of any $r$-elements of $S$ are distinct. Therefore, by Remark \ref{rmk:countbinom} we have that the set
\[ 
\Bar{\Gamma}=\lbrace n_{i_1}+\dots+n_{i_r}:i_j\in [k], j\in [r], i_1\leq\dots\leq i_r\rbrace
\]
contains $\binom{k+r-1}{r}$ distinct elements. Thus the exponents of the elements in $\Gamma$ are distinct and, since $n_{i_1}+\dots +n_{i_r}\leq hr<n$ for any $i_j\in [k], j\in [r]$, then the generating elements of $V^r$, i.e.
\[\gamma^{n_{i_1}+\dots +n_{i_r}}: i_j\in [k], j\in [r]\]
are $\fq$-linearly independent over $\fq$.\\
Therefore, $\dim_{\fq}(V^r)=\binom{k+r-1}{r}$ and $V$ is a max-span Sidon space by Theorem \ref{thm:maxspanrsidonspace}.
\end{proof}

\begin{example}
For any prime power $q$, let consider $S=\lbrace 0,1,4,16\rbrace\subseteq [16]\subseteq \mathbb{Z}$, which is a $B_3$-set in $\mathbb{Z}$. By taking $n=49>3\cdot16=48$ and $\gamma\in\F_{q^{49}}\setminus\F_{q^7}$, then by Theorem \ref{thm:maxspanrsidonspace} we have that
\[
V=\langle 1,\gamma,\gamma^4,\gamma^{16}\rangle_{\F_{q}} 
\]
is a max-span $3$-Sidon space in $\F_{q^{49}}$ of dimension $4$.
\end{example}

We will now explore another construction of max-span $r$-Sidon spaces via a polynomial setting. We start with the following auxiliary lemma.

\begin{lemma}
\label{p(a==q(a) iff p(x)=q(x)}
 Let $f_1(x),f_2(x),\ldots,f_t(x)\in\mathbb{F}_q[x]\setminus\lbrace 0\rbrace$ and let $d=\max\lbrace \deg f_i(x)\colon i\in[t]\rbrace$. Then for any $n>d$ and for any $\gamma\in\fqn$ such that $\fqn=\fq(\gamma)$, the elements $f_1(\gamma),\ldots,f_t(\gamma)$ of $\fqn$ are $\fq$-linearly independent if and only if the polynomials $f_1(x),\ldots,f_t(x)$ in $\fq[x]$ are $\fq$-linearly independent.
\begin{proof}
Suppose that the elements $f_1(\gamma),\ldots,f_t(\gamma)$ of $\fqn$ are $\fq$-linearly independent and let
\begin{equation}
  \displaystyle\sum_{i\in [t]}\alpha_if_i(x)=0   
\end{equation}
where $\alpha_i\in\mathbb{F}_q$ for every $i\in[t]$. Then, by evaluating in $\gamma$ we get
\begin{equation}
  \displaystyle\sum_{i\in [t]}\alpha_if_i(\gamma)=0   
\end{equation}
and since  $f_1(\gamma),\ldots,f_t(\gamma)\in\fqn$ are $\fq$-linearly independent, then $\alpha_i=0$ for all $i\in[t]$. \\
Conversely, suppose that the polynomials $f_1(x),\ldots,f_t(x)$ in $\fq[x]$ are $\fq$-linearly independent and let  \begin{equation}
  p(\gamma)=\displaystyle\sum_{i\in [t]}\alpha_if_i(\gamma)=0. 
\end{equation}
where $p(x)=\displaystyle\sum_{i\in [t]}\alpha_if_i(x)$. If $p(x)$ is not the zero polynomial, $\gamma$ would be the root of a polynomial of degree $\deg p(x)\leq d<n$ and this is not possible since $\fqn=\fq(\gamma)$. Therefore, $p(x)$ is the zero polynomial and so $\alpha_i=0$ for all $i\in[t]$.
\end{proof}
\end{lemma}

By using an enough large set of irreducible polynomials over the base field, we are able to construct max-span $r$-Sidon spaces.

\begin{theorem}
\label{thm:maxspanpolinomiali}
Let $k,r$ be integers such that $1<r<k$. Let $\mathcal{I}:=\lbrace p_{i_1,\dots,i_r}(x):i_j\in [k], j\in [r],i_1\leq\dots\leq i_r\rbrace$ be a set of $\binom{k+r-1}{r}$ distinct monic irreducible polynomials over $\fq$ and let $\Delta=\max\lbrace \deg(p_{i_1,\dots,i_r}(x)):p_{i_1,\dots,i_r}(x)\in\mathcal{I}\rbrace$. For any $j\in [k]$, let 
\begin{equation}
f_j(x):= \displaystyle\prod_{\substack{(i_1,\dots,i_r)\in [k]\times\dots\times [k]:\\
i_1\leq\dots\leq i_r,\\ i_h\neq j\,\,\,\forall h\in[r] }} p_{i_1,\dots,i_r}(x),
\end{equation}
and, for $n>r\Delta\cdot\binom{k+r-2}{r}$ and $\gamma\in\mathbb{F}_{q^n}$ such that $\gamma$ does not belong to any proper subfield of $\mathbb{F}_{q^n}$. Then $V:= \langle\lbrace f_j(\gamma)\rbrace_{j\in [k]}\rangle_{\fq}$ is a max-span $r$-Sidon space.
\end{theorem}
\begin{proof}
First, let notice that $V$ has dimension $k$ over $\mathbb{F}_q$. Indeed, let observe that since the maximum degree of any polynomial in $\mathcal{I}$ is $\Delta$, then by Remark \ref{rmk:countbinom} it follows that
\begin{equation}
    \begin{aligned}
\deg(f_s(x))&=\deg\left(\displaystyle\prod_{i_1,\dots,i_r\neq s}p_{i_1,\dots,i_r}(x)\right)= \displaystyle\sum_{\substack{i_1,\dots,i_r\in [k]\times\dots\times [k],\\ i_1\leq\dots\leq i_r,\\ i_h\neq s\,\,\,\forall h \in [r]}}\deg(p_{i_1,\dots,i_r}(x))\\
&\leq \binom{k-1+r-1}{r}\Delta<\frac{n}{r}<n  
    \end{aligned}
\end{equation}
for all $s\in[k]$.\\
Hence, since $\gamma$ does not belong to any proper subfield of $\fqn$, Lemma \ref{p(a==q(a) iff p(x)=q(x)} implies that the elements $f_1(\gamma),\ldots,f_k(\gamma)$ of $\fqn$ are $\fq$-linearly independent if and only if the polynomials $f_1(x),\ldots,f_k(x)$ in $\fq[x]$ are $\fq$-linearly independent. Then, suppose that
\[
\displaystyle\sum_{j\in [k]}\alpha_jf_j(x)=0.
\]
Notice that for any $s\in[k]$, $p_{s,\dots,s}(x)$ is an irreducible factor of the right hand of the equation but not of the left hand because by definition $\gcd(f_s(x),p_{s,\dots,s}(x))=1$ and $\gcd(f_j(x),p_{s,\dots,s}(x))=p_{s,\dots,s}(x)$ for all $j,s\in [k], j\neq s$. Then we get that $\alpha_s=0$ for any $s\in[k]$. Therefore the polynomials $\lbrace f_j(x): j\in[k]\rbrace$ are linearly independent and so by Lemma \ref{p(a==q(a) iff p(x)=q(x)} $\dim_{\fq}(V)=k$.\\
In order to prove the statement, by Theorem \ref{thm:maxspanrsidonspace}, it is sufficient to show that $\dim_{\fq}(V^r)=\binom{k+r-1}{r}$. Let consider the elements
\[ 
 V^r=\langle  f_{i_1}(\gamma)\cdots f_{i_r}(\gamma):i_1,\dots,i_r\in [k], i_1\leq\ldots\leq i_r \rangle.
 \]
 We only have to prove that these elements are linearly independent over $\mathbb{F}_q$.\\ 
Let notice that, since $\frac{n}{r}>\max\lbrace \deg(f_j(x)):j\in [k]\rbrace$, then by Remark \ref{rmk:countbinom} we have that
\begin{equation}
\begin{aligned}
\deg\left(\displaystyle\sum_{\substack{i_1,\dots,i_r\in [k],\\ i_1\leq\dots\leq i_r}} \alpha_{i_1,\dots,i_r}f_{i_1}(x)\cdots f_{i_r}(x)\right)
&\leq \max\left\lbrace \deg(f_{i_1}(x)\cdots f_{i_r}(x))\right\rbrace_{i_1,\dots,i_r\in [k], i_1\leq \dots \leq i_r} \\
&= \max\left\lbrace \displaystyle\sum_{j\in[r]}{\deg(f_{i_j}(x))}\right\rbrace_{i_j\in [k],j\in[r], i_1\leq\dots\leq i_r}\\
&\leq r \max\lbrace \deg(f_{i_1}(x)),\dots, \deg(f_{i_r}(x))\rbrace_{\substack{i_j\in [k],j\in[r],\\i_1\leq\dots\leq i_r},}\\
&\leq r \Delta\\
&<n.
\end{aligned}
\end{equation}
Therefore, since $\gamma$ does not belong to any proper subfield of $\fqn$, then by Lemma \ref{p(a==q(a) iff p(x)=q(x)}, it follows that the chosen generating elements of $V^r$ are linearly independent over $\mathbb{F}_q$ if and only if the related polynomials of $ f_{i_1}(x)\cdots f_{i_r}(x)$ are linearly independent over $\mathbb{F}_q$.\\
Thus, suppose that
\begin{equation}
\label{equation F_x lin indip}
\displaystyle\sum_{\substack{i_1,\dots,i_r\in [k],\\
i_1\leq\dots\leq i_r}} \alpha_{i_1,\dots,i_r}f_{i_1}(x)\cdots f_{i_r}(x)=0
\end{equation} 
for some $\alpha_{i_1,\dots,i_r}\in\mathbb{F}_q$. \\
By taking \eqref{equation F_x lin indip} modulo $p_{s,\dots,s}(x)$ for all $s\in [k]$, it results
\[ 
\alpha_{s,\dots,s}(f_s(x))^r\equiv 0\quad \mod(p_{s,\dots,s}(x))
\]
because $p_{s,\dots,s}(x)$ divides $f_j(x)$ for all $j\neq s$. Since by $\gcd(f_s(x),p_{s,\dots,s}(x))=1$ by construction, it follows that $\alpha_{s,\dots,s}=0$ for all $s\in [k]$. Then Equation \eqref{equation F_x lin indip} becomes 
\begin{equation}
\label{equation F_x lin indip 2}
\displaystyle\sum_{\substack{i_1,\dots,i_r\in [k], i_1\leq\dots\leq i_r,\\ (i_1,\dots,i_r)\neq (i,\dots,i),\\ i\in [k]}} \alpha_{i_1,\dots,i_r}f_{i_1}(x)\dots f_{i_r}(x) = 0
\end{equation}
Now, for any $i_1\in[k]$ let consider \eqref{equation F_x lin indip 2} modulo $p_{i_1,s,\dots,s}(x)$ for all $s\in[k]$ such that $s>i_1$ in $[k]$, then we obtain
\[
\alpha_{i_1,s,\dots,s}f_{i_1}(x)f_s(x)^{r-1}\equiv 0\quad \mod(p_{i_1,s,\dots,s}(x)).
\]
because $p_{i_1,s,\dots,s}(x)$ with $s>i_1$ divides $f_j(x)$ for all $j\neq i_1$ and $j\neq s$ in $[k]$. However, $\gcd(f_s(x),p_{i_1,s,\dots,s}(x))=\gcd(f_{i_1}(x),p_{i_1,s,\dots,s}(x))=1$, therefore $\alpha_{i_1,s,\dots,s}=0$ for all $s\in[k]$ such that $s>i_1$.\\
By reiterating this argument, we obtain 
\[
\alpha_{i_1,i_2,\dots,i_r}=0
\]
for any $i_j\in[k]$ and $j\in[r]$ with $i_1\leq\dots\leq i_r$. Therefore the polynomials are linearly independent and thus the correspondent evaluations in $\gamma$ are linearly independent over $\fq$. By Theorem \ref{thm:maxspanrsidonspace} it follows that $V$ is a max-span $r$-Sidon space.
\end{proof}

\begin{example}
Let $q=7, k=4,r=3$ and let consider 
\begin{equation}
    \begin{aligned}
    \mathcal{I}=\lbrace &p_{111}(x)=x^2 + 4x + 1,
    p_{112}(x)=x^2 + 2x + 2,
    p_{113}(x)=x^2 + 2x + 5,\\
    &p_{114}(x)=x^2 + 4x + 6,
    p_{122}(x)=x^2 + 5x + 3,
    p_{123}(x)=x^2 + 5x + 5,\\
    &p_{124}(x)=x^2 + 4x + 5,
    p_{133}(x)=x^2 + 6x + 6,
    p_{134}(x)=x^2 + 3x + 1,\\
    &p_{144}(x)=x^2 + 3x + 6,
    p_{222}(x)=x^2 + 2,
    p_{223}(x)=x^2 + x + 6,\\
    &p_{224}(x)=x^2 + 4,
    p_{233}(x)=x^2 + x + 3,
    p_{233}(x)=x^2 + 3x + 5,
    p_{234}(x)=x^2 + 5x + 2,\\
    &p_{333}(x)=x^2 + x + 4,
    p_{334}(x)=x^2 + 1,
    p_{344}(x)=x^2 + 6x + 3,
    p_{444}(x)=x^2 + 6x + 4 \rbrace.
    \end{aligned}
\end{equation}
By MAGMA computations it can be easily proved that $\mathcal{I}$ is a set of $\binom{k+r-1}{r}=\binom{6}{3}=20$ distinct monic irreducible polynomials over $\F_7$. \\
Let 
\begin{equation}
    \begin{aligned}
f_1(x):=&(x^2+2)(x^2+x+6)(x^2+4)(x^2+x+3)(x^2+3x+5)\\
&(x^2+5x+2)(x^2+x+4)(x^2+1)(x^2+6x+3)(x^2+6x+4)\\
f_2(x):=&(x^2+4x+1)(x^2+2x+5)(x^2+4x+6)(x^2+6x+6)(x^2+3x+1)\\     &(x^2+3x+6)(x^2+x+4)(x^2+1)(x^2+6x+3)(x^2+6x+4)\\
f_3(x):=&(x^2+4x+1)(x^2+2x+2)(x^2+4x+6)(x^2+5x+3)(x^2+4x+5)\\     &(x^2+3x+6)(x^2+2)(x^2+4)(x^2+5x+2)(x^2+6x+4)\\
f_4(x):=&(x^2+4x+1)(x^2+2x+2)(x^2+2x+5)(x^2+5x+5)\\
&(x^2+4x+5)(x^2+6x+6)(x^2+x+6)(x^2+4)(x^2+3x+5)(x^2+x+4).
    \end{aligned}
\end{equation}

Let $\Delta=\max\lbrace\deg(p(x))\colon p(x)\in\mathcal{I}\rbrace=2$ and let $n=61>r\Delta\binom{k+r-2}{r}=3\cdot 2\cdot \binom{5}{3}=60$ and let $\gamma\in\F_{7^{61}}$. Then by Theorem \ref{thm:maxspanpolinomiali} we have that
\[
V:=\langle f_1(\gamma),f_2(\gamma),f_3(\gamma),f_4(\gamma)\rangle_{\F_{7}}
\]
is a max-span 3-Sidon space of dimension $4$ in $\F_{7^{61}}$.
\end{example}

\section{New constructions from scattered polynomials}\label{sec:constr}

In this section, we want to provide further and explicit constructions of $r$-Sidon spaces, by using combinatorial and algebraic objects such as scattered polynomials and scattered subspaces. In particular, we prove that in sufficiently large extensions, every scattered polynomial defines an $r$-Sidon space. This will allow us to construct many non-equivalent examples of them. Then we exactly determine the dimension over $\fq$ of the $r$-span of a Sidon space defined by a monomial in $\fqn$ where $n=kt$, $t\geq r+1$ in order to compare this dimension with the bound of Theorem \ref{thm:lowerbound}. We can conclude the section by considering the subspace defined by the trace function. We prove that this subspace is not a Sidon space and we compute the dimension of its $r$-span.\\
\\
Let start with the following definition.\\
A \textbf{linearized polynomial} (or $q$-\textbf{polynomial}) is a polynomial of the form
$$ f(x)=\sum_{i=0}^{t}a_i x^{q^i}, \quad a_i \in \F_{q^k}.$$
If $f$ is nonzero, the $q$-\textbf{degree} of $f$ will be the maximum $i$ such that $a_i \neq 0$.

The set of linearized polynomials forms an $\F_q$-algebra with the usual addition, the scalar multiplication by elements in $\fq$ and the composition of polynomials. We denote this $\F_q$-algebra by $\mathcal{L}_{k,q}$.
The elements of the quotient algebra $\tilde{\mathcal{L}}_{k,q}$ obtained from $\mathcal{L}_{k,q}$  modulo the two-sided ideal $(x^{q^k}-x)$ are represented by linearized polynomials of $q$-degree less than $k$, i.e.
$$ \tilde{\mathcal{L}}_{k,q} = \left\{\sum_{i=0}^{k-1}f_i x^{q^i}, \quad f_i \in \F_{q^k} \right\}.$$
It is well-known that the above $\F_q$-algebra is isomorphic to the $\fq$-algebra of the $\fq$-linear endomorphisms of $\fqn$. 
This allows us to identify a linearized polynomial with the endomorphism it defines and so we will speak of \textbf{kernel} and \textbf{rank} of a $q$-polynomial $f$ meaning by this the kernel and rank of the corresponding endomorphism, denoted by $\ker(f)$ and $\mathrm{rk}(f)$, respectively. 

We refer to Chapter 3 in \cite{lidl1997finite} and \cite{wu2013linearized} for some references on linearized polynomials.

Sheekey in \cite{sheekey2016new} introduced the notion of scattered polynomials in order to find new examples of optimal codes in the rank metric.

\begin{definition}
A polynomial $f \in {\mathcal{L}}_{k,q}$ is said to be \textbf{scattered} if for any $a,b \in \mathbb{F}_{q^k}^*$, $f(a)/a=f(b)/b$ implies that $a$ and $b$ are $\fq$-proportional.
\end{definition}

Let $f\in\mathcal{L}_{k,q}$ and $U_f=\{(x,f(x))\colon x \in \F_{q^k}\}\subseteq \mathbb{F}_{q^k}^2$ be the $k$-dimensional $\fq$-subspace defined by the graph of $f$. The polynomial $f$ is scattered if $U$ is a \textbf{scattered} $\fq$-subspace of $\fqk^2$, namely if for every $a \in \mathbb{F}_{q^k}^*$ we have
\begin{equation}\label{eq:scattcondSa,f} 
\dim_{\fq}(S_{a,f})= 1, 
\end{equation}
where $S_{a,f}=\{ \rho\in \F_{q^k}\colon \rho(a,f(a)) \in U_f \}$  (see e.g. \cite{hscatt}).\\

Our aim now is to construct $r$-Sidon spaces from scattered polynomials.

\begin{theorem}
\label{thm:Vfgammarsidon}
Let $q$ be a prime power and let $k$ and $r$ be two positive integers. Let $n=kt$ for some $t\geq r+1$, and let $\gamma\in \fqn$ be a root of an irreducible polynomial of degree $t$ over $\fqk$. If $f(x) \in \mathcal{L}_{k,q}$ is a scattered polynomial then 
\[V_{f,\gamma}=\{u+f(u)\gamma \colon u \in \fqk\}\] 
is an $r$-Sidon space of dimension $k$.
\end{theorem}
\begin{proof}
Clearly, $\dim_{\fq}(V_{f,\gamma})=k$. Now, let
\[ a_1=u_1+f(u_1)\gamma, \ldots, a_r=u_r+f(u_r)\gamma, \]
and 
\[ b_1=v_1+f(v_1)\gamma, \ldots, b_r=v_r+f(v_r)\gamma, \]
for any nonzero $u_1,\ldots,u_r,v_1,\ldots,v_r \in \fqk$, and assume that
\[ \prod_{i=1}^r a_i = \prod_{i=1}^r b_i. \]
The above equation may be read as
\[ \prod_{i=1}^r (u_i+f(u_i)\gamma) = \prod_{i=1}^r (v_i+f(v_i)\gamma), \]
and Lemma \ref{p(a==q(a) iff p(x)=q(x)} implies that the following polynomials
\[ P_{u_1,\ldots,u_r}(x)= \prod_{i=1}^r (u_i+f(u_i)x) \]
and 
\[ P_{v_1,\ldots,v_r}(x)= \prod_{i=1}^r (v_i+f(v_i)x) \]
are equal since $1,\gamma,\ldots,\gamma^r$ are $\fq$-linearly independent. In particular, this implies that the two polynomials $P_{u_1,\ldots,u_r}(x)$ and $P_{v_1,\ldots,v_r}(x)$ have the same roots over $\fqn$ with the same multiplicities, i.e. the multisets
\[ \left\{\left\{ \frac{f(u_1)}{u_1},\ldots, \frac{f(u_r)}{u_r} \right\}\right\} \]
and 
\[ \left\{\left\{ \frac{f(v_1)}{v_1},\ldots, \frac{f(v_r)}{v_r} \right\}\right\} \]
coincide. Hence for any $i \in \{1,\ldots,r\}$ there exist $j \in \{1,\ldots,r\}$ such that
\[ \frac{f(u_i)}{u_i}=\frac{f(v_j)}{v_j}, \]
and since $f(x)$ is scattered then we have that $u_i$ and $v_j$ are $\fq$-proportional and so $u_i=\rho_{i,j} v_j$ for some $\rho_{i,j} \in \fq^*$. Then
\[ a_i=u_i+f(u_i)\gamma=\rho_{i,j}(v_j+f(v_j)\gamma)=\rho_{i,j} b_j \] 
and the assertion is then proved.
\end{proof}

We will see now that the above theorem allows us to construct many non-equivalent examples.\\
Let notice that if we have two $\fq$-subspaces of $\fqn$ of type 
\[
V_{U,\gamma}=\lbrace u+u'\gamma: (u,u')\in U\rbrace
\]
\[
V_{W,\xi}=\lbrace w+w'\xi: (w,w')\in W\rbrace
\]
where $\gamma$ and $\xi$ are generating elements of $\fqn$ over $\fq$ and $U,V$ are $\fq$-subspaces of $\fqk^2$, then Theorem 6.2 in \cite{CPSZSidon} characterizes the equivalence among the orbits defined by them. In particular, in the case in which $U,W$ are the graphs of two non-zero $q$-polynomials $f,g\in\mathcal{L}_{k,q}$, i.e.
\[
U=\lbrace (u,f(u)):u\in\fqk\rbrace\\
\]
and
\[
W=\lbrace (u,g(u)):u\in\fqk\rbrace
\]
then \cite[Theorem 6.2]{CPSZSidon} can be read as follows.

\begin{theorem}
\label{thm:equivalenzasidon}
Let $k$ and $n$ be two positive integers such that $k \mid n$.
Let $f,g\in\mathcal{L}_{k,q}$ be non-zero $q$-polynomials and let $U=\lbrace (u,f(u)):u\in\fqk\rbrace$ and $W=\lbrace (u,g(u)):u\in\fqk\rbrace$ be two $k$-dimensional $\fq$-subspaces of $\mathbb{F}_{q^k}^2$. Consider 
\begin{center}
    $V_{U,\gamma}=\lbrace u+f(u)\gamma: u \in \F_{q^k}\rbrace$\\
    $V_{W,\xi}=\lbrace w+g(w)\xi: w \in \F_{q^k}\rbrace$
\end{center}
where $\gamma,\xi\in\fqn$ are such that $\lbrace 1,\gamma\rbrace$ and $\lbrace 1,\xi\rbrace$ are $\fqk$-linearly independent. Then $V_{U,\gamma}$ and $V_{W,\xi}$ are semilinearly equivalent under the action of $(\lmb,\sigma)\in\fqn^*\times \mathrm{Aut}(\fqn)$ if and only if there exists $A=\left(\begin{aligned}
    \begin{matrix}
        c & d \\
        a & b 
    \end{matrix}
\end{aligned}\right) \in \mathrm{GL}(2,\fqk)$ such that $\xi=\frac{a+b\gamma^\sigma}{c+d\gamma^\sigma}$, $\lmb=\frac{1}{c+d\gamma^\sigma}$ and $U^{\sigma}=\{w A \colon w \in W\}=W \cdot A$.
\end{theorem}

As an immediate consequence, we have that if $U$ and $W$ are not $\GammaL(2,q^k)$-equivalent, then the $\fq$-subspaces $V_{f,\gamma}$ and $V_{g,\xi}$ are not semilinear equivalent.\\
We resume in Table \ref{scattpoly} the list of the known examples of scattered polynomials belonging to $\mathcal{L}_{k,q}$. Different entries of the table correspond to $\mathrm{\Gamma L}(2,q^k)$-inequivalent subspaces $U_f=\{(x,f(x)) \,:\, x\in \fqk\}$. If $f$ and $g$ are two linearized polynomials belonging to different entries of Table \ref{scattpoly}, then by Theorem 6.2 in \cite{CPSZSidon} we have that for every $\sigma \in \mathrm{Aut}(\fq)$ it results that
\[ \mathrm{Orb}(V_{f,\gamma})\ne \mathrm{Orb}(V_{g,\gamma}^\sigma), \]
that is, the corresponding codes are inequivalent.

\begin{table}[htp]
\centering
\tabcolsep=0.2 mm
\begin{tabular}{|c|c|c|c|c|c|}
\hline
\hspace{0.5cm} & \hspace{0.2cm}$k$\hspace{0.2cm} & $f(x)$ & \mbox{Conditions} & \mbox{References} \\ \hline
i) & & $x^{q^s}$ & $\gcd(s,k)=1$ & \cite{blokhuis2000scattered} \\ \hline
ii) & & $x^{q^s}+\delta x^{q^{s(k-1)}}$ & $\begin{array}{cc} \gcd(s,k)=1,\\ \mathrm{N}_{q^k/q}(\delta)\neq 1 \end{array}$ & \cite{lunardon2001blocking,lavrauw2015solution}\\ \hline
iii) & $2\ell$ & $\begin{array}{cc}x^{q^s}+x^{q^{s(\ell-1)}}+\\ \delta^{q^\ell+1}x^{q^{s(\ell+1)}}+\delta^{1-q^{2\ell-1}}x^{q^{s(2\ell-1)}}\end{array}$ & $\begin{array}{cc} q \hspace{0.1cm} \text{odd}, \\ \mathrm{N}_{q^{2\ell}/q^\ell}(\delta)=-1,\\ \gcd(s,\ell)=1 \end{array}$ & $\begin{array}{cc}\text{\cite{BZZ,longobardi2021large}}\\\text{\cite{longobardizanellascatt,neri2022extending,ZZ}}\end{array}$\\ \hline
iv) & $6$ & $x^q+\delta x^{q^{4}}$  &  $\begin{array}{cc} q>4, \\ \text{certain choices of} \, \delta \end{array}$ & \cite{csajbok2018anewfamily,bartoli2021conjecture,polverino2020number} \\ \hline
v) & $6$ & $x^{q}+x^{q^3}+\delta x^{q^5}$ & $\begin{array}{cccc}q \hspace{0.1cm} \text{odd}, \\ \delta^2+\delta =1, \\ q \hspace{0.1cm} \text{even}, \\ \text{certain choices of} \, \delta\,\, \&\,\, q  \end{array}$
 & \cite{csajbok2018linearset,MMZ,newnew} \\ \hline
vi) & $8$ & $x^{q}+\delta x^{q^5}$ & $\begin{array}{cc} q\,\text{odd},\\ \delta^2=-1\end{array}$ & \cite{csajbok2018anewfamily} \\ \hline
\end{tabular}
\caption{Known examples of scattered polynomials $f$}
\label{scattpoly}
\end{table}

\begin{remark}
Choosing $n=k(r+1)$ and $f(x)=x^q$, from Theorem \ref{thm:Vfgammarsidon} we get that
\[ V_{x^q,\gamma}=\{ u^q + u \gamma \colon u \in \fqk \} \]
is an $r$-Sidon space.
Note that 
\[ \gamma^{-1}V_{x^q,\gamma}=\{ u+u^q\gamma^{-1} \colon u \in \fqk \} \]
is the example of $r$-Sidon space found in \cite{Roth}.
\end{remark}

In the following theorem, we provide the dimension of the $r$-span of the $r$-Sidon space $V_{f,\gamma}\subseteq\fqk(\gamma)=\fqn$ with $n=kt, t\geq r+1$, obtained by taking $f(x)=x^{q^s}$ with $\gcd(s,k)=1$. We start from the following remark.

\begin{remark}
\label{rmk:sumsubspace}
We recall that for any $U_1,U_2,\dots,U_h,W\leq_q \fqn$ it holds
\[
(U_1+U_2+\dots+U_h)W=U_1W+U_2W+\dots+U_hW.
\]
Indeed, suppose that $\dim_{\fq}(U_i)=k_i$ for any $i\in[h]$, $\dim_{\fq}(W)=k$ and $U_i=\langle u_{i,j}\colon j\in[k_i]\rangle_{\fq}$ for any $i\in[h]$ and $W=\langle w_l\colon l\in [k] \rangle_{\fq}$. Then by definition we have that
\[
U_1+U_2+\dots+U_h=\langle u_{i,j}\colon j\in[k_i], i\in[h]\rangle_{\fq}
\]
and
\[
(U_1+U_2+\dots+U_h)W=\langle u_{i,j}w_l\colon j\in[k_i], i\in[h], l\in[k]\rangle=U_1W+U_2W+\dots+U_hW.
\]
\end{remark}

\begin{theorem}
\label{thm:monomialrspan}
Let $V=V_{x^{q^s},\gamma}=\lbrace u+\gamma u^{q^s}\colon u\in\fqk\rbrace\subseteq \fqk(\gamma)=\fqn$ where $n=kt$, $t\geq r+1$ and $\gcd(s,k)=1$. Then
\[
V^r=V_{x^{q^s},\gamma}^r=\gamma\fqk\oplus\dots\oplus\gamma^{r-1}\fqk\oplus V_{x^{q^s},\gamma^{r}}
\]
and
\[
\dim_{\fq}(V_{x^{q^s},\gamma}^r)=rk
\]
for any $2\leq r\leq t-1$.\\
Moreover, denoted by 
\[
\Bar{t}=\min\lbrace s\colon V^s=V^{s+1}\rbrace,
\]
if $\mathrm{N}_{q^n/q}(\gamma)\neq (-1)^{n}$ then $\Bar{t}=t$, whereas if $\mathrm{N}_{q^n/q}(\gamma)=(-1)^{n}$, then $\Bar{t}=t+1$.
\end{theorem}
\begin{proof}
Let $r=2$ and let notice that 
\[
V^2=\langle uv+(uv^{q^s}+u^{q^s}v)\gamma+(uv)^{q^s}\gamma^2\colon u,v\in\fqk\rangle_{\fq}\subseteq \lbrace w+w^{q^s}\gamma^2+z\gamma:w,z\in\fqk\rbrace=\gamma\fqk+V_{x^{q^s},\gamma^2}.
\]
Moreover, since $t\geq 3$, the elements $1,\gamma,\gamma^2$ are $\fqk$-linearly independent, thus $\dim_{\fq}(\gamma\fqk+V_{x^{q^s},\gamma^2})=\dim_{\fq}(\gamma\fqk)+\dim_{\fq}(V_{x^{q^s},\gamma^2})=2k$. Then $\dim_{\fq}(V^2)\leq 2k$. On the other hand, by Theorem \ref{lowerboundSidon}, we have $\dim_{\fq}(V^2)\geq 2k$. Then it follows that \[
\dim_{\fq}(V^2)=2k
\]
and hence
\[
V^2=\gamma\fqk\oplus V_{x^{q^s},\gamma^2}.
\]
Now, let $r\geq 3$ and $t\geq r+1$ and let proceed by induction on $r$. Suppose that 
\begin{equation}
    \label{eq:inductiononr}
    V^{r-1}=\gamma\fqk+\dots+\gamma^{r-2}\fqk+V_{x^{q^s},\gamma^{r-1}}
\end{equation}
and hence
\begin{equation}
    \label{eq:inductiondimension}
    \dim_{\fq}(V^{r-1})=(r-1)k.
\end{equation}
Let consider $\fqk V=\langle a u+a u^{q^s}\gamma\colon u\in\fqk, a\in \fqk\rangle_{\fq}$. Note that $\fqk V$ is an $\fqk$-subspace such that
\[
V\subseteq \fqk V\subseteq \fqk\oplus\gamma\fqk 
\]
hence either $\fqk V=V$ or $\fqk V=\fqk\oplus\gamma\fqk$. If $\fqk V=V$ then $V$ is $\fqk$-linear and this is not true since $\gcd(s,k)=1$. Therefore 
\begin{equation}
\label{eq:fqkV}
  \fqk V=\fqk\oplus\gamma\fqk  
\end{equation}
and in particular 
\begin{equation}
\label{eq:gammafqkVs}
  \gamma^l \fqk V=\gamma^l\fqk\oplus\gamma^{l+1}\fqk  
\end{equation}
for any $1\leq l\leq r-2$. \\
Moreover, by definition we have that
\begin{equation}
\label{eq:Vr-1V}
    \begin{aligned}
    (V_{x^{q^s},\gamma^{r-1}})V+\gamma\fqk+\gamma^{r-1}\fqk&=\langle (u+u^{q^s}\gamma^{r-1})(v+v^{q^s}\gamma)\colon u,v\in\fqk \rangle_{\fq}+\gamma\fqk+\gamma^{r-1}\\
    &=\langle uv+uv^{q^s}\gamma+vu^{q^s}\gamma^{r-1}+(uv)^{q^s}\gamma^r\colon u,v\in\fqk\rangle_{\fq}+\gamma\fqk+\gamma^{r-1}\fqk \\
    &=\langle a + a^{q^s}\gamma^r\colon a\in\fqk\rangle_{\fq}+\gamma\fqk+\gamma^{r-1}\fqk.
    \end{aligned}
\end{equation}
Therefore, by Remark \eqref{rmk:sumsubspace} and by Equations \eqref{eq:inductiononr}, \eqref{eq:gammafqkVs} and \eqref{eq:Vr-1V} we get that
\begin{equation}
\label{eq:v3}
\begin{aligned}
    V^r=(V^{r-1})V&=(V_{x^{q^s},\gamma^{r-1}}+\gamma\fqk+\dots+\gamma^{r-2}\fqk)V\\
            &= V_{x^{q^s},\gamma^{r-1}}V+\gamma\fqk V+\dots+\gamma^{r-2}\fqk V\\
            &= \langle a + a^{q^s}\gamma^r\colon a\in\fqk\rangle_{\fq}+\gamma\fqk+\gamma^2\fqk+\dots+\gamma^{r-2}\fqk+\gamma^{r-1}\fqk\\
            &=V_{x^{q^s},\gamma^{r}}+ \gamma\fqk+\dots+\gamma^{r-1}\fqk
\end{aligned}
\end{equation}
Since $1,\gamma,\dots,\gamma^{r}$ are $\fqk$-linearly independent, we have that
\[
\dim_{\fq}(V^r)=rk.
\]
Now, if we consider $r=t-1$, by Equation \eqref{eq:fqkV} and Remark \ref{rmk:sumsubspace} we get
\begin{equation}
\label{eq:vt}
    \begin{aligned}
     V^{t}=(V^{t-1})V&=(\gamma\fqk+\dots+\gamma^{t-2}\fqk+V_{x^{q^s},\gamma^{t-1}})V\\
     &=\gamma\fqk V+\dots+\gamma^{t-2}\fqk V+V_{x^{q^s},\gamma^{t-1}}V\\
     &=\gamma(\fqk+\gamma\fqk)+\dots+\gamma^{t-2}(\fqk+\gamma\fqk)+\langle (u+u^{q^s}\gamma)(v+v^{q^s}\gamma^{t-1})\colon u,v\in\fqk\rangle_{\fq}\\
     &=\gamma\fqk+\gamma^2\fqk+\dots+\gamma^{t-1}\fqk+\lbrace w+w^{q^s}\gamma^t\colon w\in\fqk\rbrace.\\
    \end{aligned}
\end{equation}
Let $p(x)=a_0+a_1x+\dots+a_{t-1}x^{t-1}+x^{t}$ be the minimal polynomial of $\gamma$ over $\fqk$, then Equation \eqref{eq:vt} becomes
\begin{equation}
\label{eq:vta0}
    \begin{aligned}
     V^{t}&=\gamma\fqk+\dots+\gamma^{t-1}\fqk+\lbrace w+w^{q^s}\gamma^t\colon w\in\fqk\rbrace\\
     &=\gamma\fqk+\dots+\gamma^{t-2}\fqk+\lbrace w-w^{q^s}a_0\colon w\in\fqk\rbrace
    \end{aligned}
\end{equation}
Note that $\lbrace w-wa_0\colon w\in\fqk\rbrace=\mathrm{Im}(\phi_{a_0})$ where
\[
\begin{aligned}
    \phi_{a_0}\colon \fqk &\to \fqk \\
                    w&\mapsto w-a_0w^{q^s}
\end{aligned}       
\]
and this map is invertible if and only if $\mathrm{N}_{q^k/q}(a_0)\neq 1$. Moreover, since $a_0$ is the constant term of $p(x)$, we also have that
\[
a_0=(-1)^t\gamma^{q^{(t-1)k}+\dots+q^k+1}=(-1)^t\gamma^{\frac{q^{kt}-1}{q^k-1}}=(-1)^t\gamma^{\frac{q^n-1}{q^k-1}}
\]
and so 
\[ \mathrm{N}_{q^{k}/q}(a_0)=(-1)^{kt}\mathrm{N}_{q^{n}/q}(\gamma). \]
In particular,
\[  \mathrm{N}_{q^{k}/q}(a_0)=1 \text{ if and only if } \mathrm{N}_{q^{n}/q}(\gamma)=(-1)^{kt} \]
and so $\phi_{a_0}$ is invertible if and only if $\mathrm{N}_{q^{n}/q}(\gamma)\neq (-1)^{kt}$. \\
Hence, if $\mathrm{N}_{q^{n}/q}(\gamma)\neq (-1)^{kt}$ then $\mathrm{Im}(\phi_{a_0})=\fqk$ and so from Equation \eqref{eq:vta0} we get
\[
V^t=\fqk+\gamma\fqk+\dots+\gamma^{t-1}\fqk=\fqn.
\]
and clearly $\Bar{t}=t$.\\
Whereas, if $\mathrm{N}_{q^{n}/q}(\gamma)=(-1)^{kt}$, since $\phi_{a_0}(x)=x+x^{q^s}a_0$ is a $q^s$-polynomial of $q^s$-degree equal to $1$ (see e.g. \cite{gow}), we have that $\dim_{\fq}(\mathrm{Im}(\phi_{a_0}))=k-1$ and so by Equation \eqref{eq:vta0} it follows that $\dim_{\fq}(V^t)=(t-1)k+k-1=tk-1=n-1$. Now, if $V^{t+1}=V^t$ then, once we fix $\lbrace v_1,\dots,v_k\rbrace$ an $\fq$-basis of $V$, we have that
\[
V^{t+1}=v_1V^t+\dots+v_kV^t
\]
thus $V^t$ is $\fq(V)$-linear. Since $\gcd(n,n-1)=1$ by Lemma \ref{lem:hv=fq} we have a contradiction. Therefore, $V^{t+1}=\fqn$ and so $\Bar{t}=t+1$.
\end{proof}

In order to compare the dimension of the $r$-span of $V_{x^{q^s},\gamma}$ with the bound of Theorem \ref{thm:lowerbound} we need the following results.

\begin{lemma}
    \label{lem:stabpowergamma}
Let $\gamma\in\fqn$ be a generating element of $\fqn$ over $\fqk$ and let $n=kt$. Denote by
\[
U:=\fqk +\gamma\fqk+\dots+\gamma^{l-1}\fqk+\gamma^l\fqk
\]
for some $l<t-1$, then
\[
H(U)=\fqk.
\]
\end{lemma}
\begin{proof}
Clearly $\fqk\subseteq H(U)$. Let $\alpha\in H(U)\setminus\lbrace 0\rbrace$, then $\alpha U=U$, i.e.
\begin{equation}
\label{eq:alfau=u}
\fqk +\gamma\fqk+\dots+\gamma^{l-1}\fqk+\gamma^{l}\fqk=\alpha \fqk + \alpha\gamma\fqk+\dots+\alpha\gamma^{l-1}\fqk+\alpha\gamma^{l}\fqk.
\end{equation}
In particular, $\alpha\in U$ then there exist $a_0,a_1,\dots,a_{l}\in\fqk$ such that
\[
    \alpha=a_0+a_1\gamma+\dots+a_{l}\gamma^{l}
\]
then
\begin{equation}
\label{eq:system1}
    \begin{cases}
    \alpha=a_0+a_1\gamma+\dots+a_{l}\gamma^{l}\\
    \alpha\gamma=a_0\gamma+a_1\gamma^2+\dots+a_{l}\gamma^{l+1}\\
    \alpha\gamma^2=a_0\gamma^2+a_1\gamma^3+\dots+a_{l-1}\gamma^{l+1}+a_l\gamma^{l+2}\\
    \vdots\\
    \alpha\gamma^l=a_0\gamma^{l}+a_1\gamma^{l+1}+\dots+a_{l-1}\gamma^{2l-1}+a_l\gamma^{2l}
    \end{cases}
\end{equation}
By Equation \eqref{eq:alfau=u} we have that also $\alpha\gamma,\alpha\gamma^2,\dots,\alpha\gamma^l\in U$. Since $\alpha\gamma\in U$ and $1,\gamma,\dots,\gamma^l,\gamma^{l+1}$ are $\fqk$-linearly independent by the second equation of System \eqref{eq:system1} we get that $a_l=0$. Then, since $a_l=0$ and $1,\gamma,\dots,\gamma^{l+1}$ are $\fqk$-linearly independent and $\alpha\gamma^2\in U$, by the third equation of System \eqref{eq:system1}, $a_{l-1}=0$. By reiterating this argument we get that also $a_{l-2}=\dots=a_1=0$. Thus $\alpha=a_0\in\fqk$.
\end{proof}

\begin{theorem}
    \label{thm:stabilizzmonomio}
Let $V=V_{x^{q^s},\gamma}=\lbrace u+\gamma u^{q^s}\colon u\in\fqk\rbrace\subseteq \fqk(\gamma)=\fqn$ where $n=kt$, $t\geq r+1$ and $\gcd(s,k)=1$. Then
\[
H(V_{x^{q^s},\gamma}^r)=\fq.
\]
\end{theorem}
\begin{proof}
Let $\alpha\in H(V_{x^{q^s},\gamma}^r)\setminus\lbrace 0\rbrace$, then
\begin{equation}
\label{alfa vr =vr}
    \alpha V_{x^{q^s},\gamma}^r=V_{x^{q^s},\gamma}^r.
\end{equation}
First, let notice that $\alpha\in\fqk$. Indeed, suppose by contradiction that $\alpha\in\fqn\setminus\fqk$. By Theorem \ref{thm:monomialrspan}, Equation \eqref{alfa vr =vr} implies that
\begin{equation}
    \label{alfa vr=vr 2}
\alpha V_{x^{q^s},\gamma}^r=\alpha\gamma\fqk\oplus\dots\oplus\alpha\gamma^{r-1}\fqk\oplus\alpha V_{x^{q^s},\gamma^r}=\gamma\fqk\oplus\dots\oplus\gamma^{r-1}\fqk\oplus V_{x^{q^s},\gamma^r}.
\end{equation}
Therefore, denoted by $W:=\gamma\fqk+\dots+\gamma^{r-1}\fqk+V_{x^{q^s},\gamma^r}$, by Theorem \ref{thm:monomialrspan} we have that
$\dim_{\fq}(W)=rk$.\\
On one hand, Equation \eqref{alfa vr=vr 2} implies that
\[
\alpha\gamma\fqk+\dots+\alpha\gamma^{r-1}\fqk\subseteq W.
\]
On the other hand, we also have that
\[
\gamma\fqk+\dots+\gamma^{r-1}\fqk\subseteq W.
\]
Note that if $\alpha\gamma\fqk+\dots+\alpha\gamma^{r-1}\fqk=\gamma\fqk+\dots+\gamma^{r-1}\fqk$ then $\alpha\in H(\gamma\fqk+\dots+\gamma^{r-1}\fqk)$. By Lemma \ref{lem:stabpowergamma} and Remark \ref{rmk:h(v)=h(gamma-1v)} it follows that $\alpha\in\fqk$, a contradiction. Thus, 
\begin{equation}
    \label{eq:w}
   W=\alpha\gamma\fqk+\dots+\alpha\gamma^{r-1}\fqk+\gamma\fqk+\dots+\gamma^{r-1}\fqk.
\end{equation}
By definition $V_{x^{q^s},\gamma^r}\subseteq \fqk+\gamma^r\fqk$ and by Equation \eqref{alfa vr=vr 2}, $V_{x^{q^s},\gamma^r}\subseteq W$. Therefore, by Equation \eqref{eq:w} it follows that
\[
V_{x^{q^s},\gamma^r}\subseteq (\gamma\fqk+\dots+\gamma^{r-1}\fqk+\alpha\gamma\fqk+\dots+\alpha\gamma^{r-1}\fqk)\cap(\fqk+\gamma^r\fqk).
\]
Let notice that
\[
1\leq\dim_{\fqk}((\gamma\fqk+\dots+\gamma^{r-1}\fqk+\alpha\gamma\fqk+\dots+\alpha\gamma^{r-1}\fqk)\cap(\fqk+\gamma^r\fqk))\leq 2
\]
Moreover, if $\dim_{\fqk}((\gamma\fqk+\dots+\gamma^{r-1}\fqk+\alpha\gamma\fqk+\dots+\alpha\gamma^{r-1}\fqk)\cap(\fqk+\gamma^r\fqk))=2$, then
\[
\fqk+\gamma^r\fqk\subseteq \gamma\fqk+\dots+\gamma^{r-1}\fqk+\alpha\gamma\fqk+\dots+\alpha\gamma^{r-1}\fqk
\]
i.e. $\fqk+\gamma\fqk+\dots+\gamma^r\fqk\subseteq W$, a contradiction since $\dim_{\fq}(W)=rk$.
Then $\dim_{\fqk}((\gamma\fqk+\alpha\gamma\fqk+\dots+\alpha\gamma^{r-1}\fqk)\cap(\fqk+\gamma^r\fqk))=1=\dim_{\fqk}(V_{x^{q^s},\gamma^r})$ and so
\[
V_{x^{q^s},\gamma^r}=\xi\fqk
\]
for some $\xi\in\fqn$. This implies that $V_{x^{q^s},\gamma^r}$ is $\fqk$-linear and this is not true since $\gcd(k,s)=1$. Therefore $\alpha\in\fqk$.\\
Now, let prove that $\alpha\in\fq$. Since $\alpha\in\fqk$ and $ 1,\gamma,\dots,\gamma^r$ are $\fqk$-linearly independent, Equation \eqref{alfa vr =vr} implies that
\[
\alpha V_{x^{q^s},\gamma^r}=V_{x^{q^s},\gamma^r}.
\]
This means that for any $u\in\fqk$ there exists $v\in\fqk$ such that
\[
\alpha u + \gamma^r\alpha u^{q^s}=v+\gamma^rv^{q^s}.
\]
This implies that
$\alpha u=v$ and $\alpha u^{q^s}=v^{q^s}=\alpha^{q^s}u^{q^s}$. Thus, $\alpha^{q^s}=\alpha$ and since $\gcd(k,s)=1$ then $\alpha\in\fq$.
\end{proof}

By MAGMA computations, discussed in the next section, it seems that the lower bound provided in Theorem \ref{thm:lowerbound} is not tight. However, by applying Theorem \ref{thm:lowerbound} to the case of $r$-Sidon spaces obtained by the monomial, as in Theorem \ref{thm:monomialrspan}, by Theorem \ref{thm:stabilizzmonomio} we get that $\dim_{\fq}(V_{x^{q^s},\gamma}^r)\geq rk-(r-2)h(V_{x^{q^s},\gamma}^r)=rk-(r-2)$ . Therefore, Theorem \ref{thm:monomialrspan} shows that $V_{x^{q^s},\gamma}^r$ has dimension very close to the lower bound given in Theorem \ref{thm:lowerbound}.\\
We have shown in Theorem \ref{thm:lowerboundprime} that in the case of prime degree extensions for any $r$-Sidon space $V$ we have that $\dim_{\fq}(V^r)\geq r\dim_{\fq}(V)$. We believe that the same lower bound holds also in the case of composite degree extensions, but we weren't able to prove it up to now. If this is true, Theorem \ref{thm:monomialrspan} would provide examples of $r$-Sidon spaces with $r$-span of minimum dimension, as it happens for the case $r=2$, see \cite{Roth}. \\

Now, we want to provide other examples of the type $V_{f,\gamma}$ such that $\dim_{\fq}(V_{f,\gamma}^r)=rk$, as happens for the case of the monomial $x^{q^s}$ and not necessarily Sidon spaces. We will consider the subspace obtained by taking $f(x)=\mathrm{Tr}_{q^k/q}(x)$. In order to do so, let start with the following remark and lemma.

\begin{remark}
\label{rmk:nucleotraccia}
Recall that 
\[
\begin{aligned}
\mathrm{Tr}_{q^k/q} \colon \mathbb{F}_{q^k}&\to \mathbb{F}_{q}\\
x&\mapsto  x+ x^q+\ldots+x^{q^{k-1}}
\end{aligned}
\]
is an $\fq$-linear map of $\fqk$ in $\fq$ and $\dim_{\fq}(\ker(\mathrm{Tr}_{q^k/q}))=k-1$.
\end{remark}

\begin{lemma}
    \label{thm:traccianosidon}
Let $V=\lbrace u+\mathrm{Tr}_{q^k/q}(u)\gamma\colon u\in\fqk\rbrace\subseteq \fqk(\gamma)=\fqn$. Then
\[
\dim_{\fq}(V\cap\alpha V)= k-2.
\]
for any $\alpha\in\fqk\setminus\fq$.
In particular, if $k>3$ then $V$ is not a Sidon space.
\end{lemma}
\begin{proof}
Let $a\in V\cap\alpha V$ and let $\alpha\in\fqk\setminus\fq$ then 
\[
\alpha V=\lbrace \alpha u+\alpha \mathrm{Tr}_{q^k/q}(u)\gamma\colon u\in\fqk\rbrace_{\fq}.
\]
Since $a\in V\cap \alpha V$ then there exist $u,v\in\fqk$ such that 
\[
a=u+\mathrm{Tr}_{q^k/q}(u)\gamma=\alpha v +\gamma \alpha \mathrm{Tr}_{q^k/q}(v).
\]
Since $1,\gamma$ are $\fqk$-linearly independent, this implies that
\[
u=\alpha v 
\]
and
\[
\mathrm{Tr}_{q^k/q}(u)=\alpha \mathrm{Tr}_{q^k/q}(v).
\]
Let notice that $\mathrm{Tr}_{q^k/q}(u),\mathrm{Tr}_{q^k/q}(v)\in\fq$, and since $\alpha\notin\fq$, this implies that $\mathrm{Tr}_{q^k/q}(u)=\mathrm{Tr}_{q^k/q}(v)=0$. Therefore $u,v\in \ker(\mathrm{Tr}_{q^k/q})$. \\
Conversely, let consider $u\in \ker(\mathrm{Tr}_{q^k/q})\cap\alpha \ker(\mathrm{Tr}_{q^k/q})$. Then 
\[
u=\alpha v
\]
for some $v\in\ker(\mathrm{Tr}_{q^k/q})$. Then, since $\alpha\in\fqk$,  
\[
u+\gamma\mathrm{Tr}_{q^k/q}(u)=\alpha v+\mathrm{Tr}_{q^k/q}(\alpha v))=\alpha(v+\gamma\mathrm{Tr}_{q^k/q}(v))\]
and so $u=\alpha v \in V\cap \alpha V$.\\
This implies that 
\[
V\cap \alpha V= \ker(\mathrm{Tr}_{q^k/q})\cap \alpha \ker(\mathrm{Tr}_{q^k/q}).\]
Moreover, by Remark \ref{rmk:nucleotraccia}, $\dim_{\fq}(\ker(\mathrm{Tr}_{q^k/q}))=k-1$ and $\ker(\mathrm{Tr}_{q^k/q})\subseteq \fqk$, therefore
\[
k\geq \dim_{\fq}(\ker(\mathrm{Tr}_{q^k/q})+\alpha \ker(\mathrm{Tr}_{q^k/q})=2(k-1)-\dim_{\fq}(\ker(\mathrm{Tr}_{q^k/q})\cap \alpha \ker(\mathrm{Tr}_{q^k/q}))
\]
and so
\[
\dim_{\fq}(\ker(\mathrm{Tr}_{q^k/q})\cap \alpha \ker(\mathrm{Tr}_{q^k/q}))\geq k-2.
\]
Also, since $\alpha\notin\fq$, $\alpha\ker(\mathrm{Tr}_{q^k/q})\neq\ker(\mathrm{Tr}_{q^k/q})$ and hence 
\[
\dim(V\cap\alpha V)=\dim_{\fq}(\ker(\mathrm{Tr}_{q^k/q})\cap \alpha \ker(\mathrm{Tr}_{q^k/q}))= k-2.
\]
In particular, this means that there exists $\alpha\notin\fq$ such that $\dim_{\fq}(V\cap\alpha V)\geq k-2$, and if $k>3$, by Theorem \ref{lem:charSidon}, this implies that $V$ is not a Sidon space.

\end{proof}


In Theorem \ref{thm:monomialrspan} we have proved that $\dim_{\fq}(V_{x^{q^s},\gamma})=rk$ and by Theorem \ref{thm:Vfgammarsidon} we also know that $V_{x^{q^s},\gamma}$ is an $r$-Sidon space in $\F_{q^{kt}}$ for any $t\geq r+1$. Now, let consider $V_{Tr_{q^k/q},\gamma}$. By Theorem \ref{thm:traccianosidon} we have seen that if $k>3$ then $V_{\mathrm{Tr}_{q^k/q},\gamma}$ is not a Sidon space in $\fqn$, and in particular, by Lemma \ref{lem:boxprop}, $V_{\mathrm{Tr}_{q^k/q},\gamma}$ is not an $r$-Sidon space for any $r\geq 2$. However, in the following theorem we will show that, also for subspaces of the type $V_{\mathrm{Tr}_{q^k/q},\gamma}$, the $r$-span of $V_{\mathrm{Tr}_{q^k/q},\gamma}$ has dimension equal to $rk$, as happens for $V_{x^{q^s},\gamma}$.

\begin{theorem}
\label{thm:tracciadimensionerspan}
Let $V=V_{\mathrm{Tr}_{q^k/q},\gamma}=\lbrace u+\mathrm{Tr}_{q^k/q}(u)\gamma\colon u\in\fqk\rbrace\subseteq \fqk(\gamma)=\fqn$ and $\frac{n}{k}=t\geq r+1$. Then
\[
V^r=\fqk+\fqk U+\dots+\fqk U^{r-2}+WU^{r-1}+U^r
\]
where $W:=\ker(\mathrm{Tr}_{q^k/q})$ and $U:=\lbrace a+\gamma \mathrm{Tr}_{q^k/q}(a)\rangle_{\fq}$ where $a\in\fqk$ and $\mathrm{Tr}_{q^k/q}(a)\neq 0$. In particular,
\[
\dim_{\fq}(V^r)=rk.
\]
\end{theorem}
\begin{proof}
Let $V=V_{\mathrm{Tr}_{q^k/q},\gamma}=\lbrace u+\mathrm{Tr}_{q^k/q}(u)\gamma\colon u\in\fqk\rbrace$.  By Remark \ref{rmk:nucleotraccia}, we know that $\dim_{\fq}(\ker(\mathrm{Tr}_{q^k/q}))=k-1$. Then let 
\[
W:=\ker(\mathrm{Tr}_{q^k/q})=\langle u_1,\dots,u_{k-1}\rangle_{\fq}\subseteq \fqk.
\]
Therefore
\[
W=\langle u_1,\dots,u_{k-1}\rangle_{\fq}\subseteq V.
\]
Moreover, there exists $a\in\fqk$ such that $\mathrm{Tr}_{q^k/q}(a)\neq 0$. Thus denote by
\[
U:=\langle a+\gamma \mathrm{Tr}_{q^k/q}(a)\rangle_{\fq},
\]
since $\dim_{\fq}(V)=k$, then
\[
V=\langle u_1,\dots,u_{k-1}\rangle_{\fq}\oplus\langle a+\gamma \mathrm{Tr}_{q^k/q}(a)\rangle_{\fq}=W\oplus U.
\]
Now, note that
\[
W^s=\fqk
\]
for any $s\geq 2$.
Indeed, let consider $\alpha,\beta \in W$ such that $\frac{\alpha}{\beta}\notin\fq$ and assume that $\alpha W=\beta W$. Then $W$ is an $\fq(\alpha/\beta)$-subspace of $\fqk$ and since $\gcd(k,k-1)=1$ by Lemma \ref{lem:hv=fq} we have a contradiction. Therefore it follows that $\dim_{\fq}(\alpha W+\beta W)=k$. Moreover, $\alpha W+\beta W\subseteq W^2\subseteq \fqk$, so we have that $W^2=\fqk$ and 
\[
W^s=\fqk 
\]
for any $s\geq 2$.    
Then we have
\begin{equation}
\begin{aligned}
V^r&=(W\oplus U)^r\\
&=\langle u_{i_1}\cdots u_{i_{r}}\colon i_j\in[k-1],j\in[r]\rangle_{\fq}+\\
&+\langle (u_{i_1}\cdots u_{i_{r-1}})(a+\gamma \mathrm{Tr}_{q^k/q}(a))\colon i_j\in[k-1],j\in[r-1]\rangle_{\fq}+ \\
& \vdots\\
&+ \langle (u_{i_1}u_{i_{2}})(a+\gamma \mathrm{Tr}_{q^k/q}(a))^{r-2}\colon i_j\in[k-1],j\in[2]\rangle_{\fq}+ \\
&+\langle u_{i_1}(a+\gamma \mathrm{Tr}_{q^k/q}(a))^r\colon i_1\in[k-1]\rangle_{\fq}+ \\
&+ \langle (a+\mathrm{Tr}_{q^k/q}\gamma)^r\rangle_{\fq}\\
&=W^r+ W^{r-1}U+\dots+ W^2U^{r-2}+ WU^{r-1}+ U^r\\
&=\fqk+\fqk U+\dots+\fqk U^{r-2}+WU^{r-1}+U^r
\end{aligned}
\end{equation}
and this is a direct sum since $1,\gamma, \dots, \gamma^{r} $ are $\fqk$-linearly independent.\\
It is clear that $\dim_{\fq}(U^j)=1$ for any $j\in[r]$ and so $\dim_{\fq}(\fqk U^j)=k$ for any $1\leq j\leq r-2$ and $\dim_{\fq}(WU^{r-1})=\dim_{\fq}(W)=k-1$.\\
Thus we get that
\begin{equation}
    \begin{aligned}
        \dim_{\fq}(V^{r})&=\dim_{\fq}(\fqk)+\dim_{\fq}(\fqk U)+\dots+\dim_{\fq}(\fqk U^{r-2})+\\
        &+\dim_{\fq}(WU^{r-1})+\dim_{\fq}(U^{r})= \\
        &=\underbrace{k+k+\dots+k}_{r-1}+(k-1)+1=rk.
    \end{aligned}
\end{equation}
\end{proof}

\begin{remark}
This polynomial approach seems very powerful to get constructions and examples. We may wonder whether multivariate linearized polynomials may give further constructions as done for rank-metric codes, see e.g. in \cite{divisible}.
\end{remark}

\section{Some computational results on $r$-Sidon spaces}\label{sec:computat}

By Lemma \ref{lem:boxprop}, we know that every $r$-Sidon space is also an $r'$-Sidon space for every $r'\leq r$. So, it is quite natural to ask whether or not a Sidon space is also an $r$-Sidon space for $r \geq 2$.
In \cite[Construction 43]{Roth}, Roth, Raviv and Tamo provided examples of $r$-Sidon spaces of dimension $k$ in $\F_{q^{k(r+1)}}$ for any prime power $q$, $k>0$ and $r\geq 2$ by considering a generating element $\gamma$ of $\fqn$ over $\fqk$. Whereas, in \cite[Construction 45]{Roth} they provided examples of $r$-Sidon spaces of dimension $k$ in $\mathbb{F}_{q^{kr}}$, for any prime power $q\geq 3$, for any integers $k>0$ and $r\geq 2$ by considering a generating element $\gamma$ of $\fqn$ over $\fqk$ such that $\mathrm{N}_{q^k/q}(\gamma)\neq 1$. \\
In particular, let notice that, by following their arguments and by Lemma \ref{lem:boxprop}, it can be easily shown the following proposition.

\begin{proposition}
\label{prop:generalizedrothraviv}
For any prime power $q$ and integers $k>0$ and $r\geq 2$, let $n=kt$ with $t\geq r+1$ and $\gcd(k,s)=1$. Let $\gamma\in\fqn$ such that $\fqk (\gamma)=\fqn$, then $V=\lbrace u+u^{q^s}\gamma\colon u\fqk\rbrace$ is an $r'$-Sidon space of $\fqn$ of dimension $k$, for any $r'\leq r$.\\
Moreover, for any prime power $q\geq 3$ and integers $k>0$ and $r\geq 2$, let $n=kt$ with $t\geq r$ and $\gcd(k,s)=1$. Let $\gamma\in\fqn$ such that $\fqk(\gamma)=\fqn$ and $\mathrm{N}_{q^k/q}(\gamma)\neq 1$. Then $V=\lbrace u+u^{q^s}\gamma\colon u\fqk\rbrace$ is an $r'$-Sidon space of $\fqn$ of dimension $k$, for any $r'\leq r$.\\
\end{proposition}

For instance, when $k=3$ and $r=2$, \cite[Construction 45]{Roth} reads as follows
\[ V=\{ u+u^q\gamma \colon u \in \F_{q^3} \}\subseteq \F_{q^6}, \]
for some $\gamma \in \F_{q^6}\setminus\F_{q^3}$.
MAGMA computations show the following.

\begin{proposition}
For any $\gamma \in \F_{2^6}\setminus \F_{2^3}$ the subspace 
\[ V=\{ u+u^2\gamma \colon u \in \F_{2^3} \}\subseteq \F_{2^6} \]
is a $2$-Sidon space but it is not a $3$-Sidon space.
\end{proposition}

Proposition \ref{prop:generalizedrothraviv} shows that when $k=3$ and $r=3$,
\[ V=\{ u+u^q\gamma \colon u \in \F_{q^3} \}\subseteq \F_{q^9} \]
is a $3$-Sidon space, for some $\gamma \in \F_{q^9}\setminus\F_{q^3}$. The next example shows that not all of the three dimensional Sidon spaces are also $3$-Sidon spaces.

\begin{proposition}\label{prop:trace}
For $q \in\{2,3\}$ there does not exist no $\gamma \in \F_{q^9}\setminus \F_{q^3}$ such that the subspace 
\[ V=\{ u+\mathrm{Tr}_{q^3/q}(u)\gamma \colon u \in \F_{q^3} \}\subseteq \F_{q^9} \]
is a $3$-Sidon space, and for any $\gamma\in\mathbb{F}_{q^9}\setminus \mathbb{F}_{q^3}$ $V$ is a Sidon space. 
\end{proposition}
\begin{proof}
The first part can be checked with MAGMA. For the second part, note that $V$ is a Sidon space, since up to equivalence (see Theorem \ref{thm:equivalenzasidon}) it can be defined by a binomial and \cite[Proposition 4.8]{CPSZSidon} guarantees the Sidon space property.
\end{proof}

\medskip

Moreover, MAGMA computations also show that, by a random check, between 14,5\% and 16,6\% of three dimensional subspaces in $\F_{2^9}$ are $3$-Sidon spaces and between 93,6\% and 95,1\% of three dimensional subspaces in $\F_{2^9}$ are $2$-Sidon spaces, showing that the $3$-Sidon space property is much stronger than the $2$-Sidon space property.\\

In Section \ref{sec:constr} we have seen that, since $x^{q^s}$ is scattered whenever $\gcd(k,s)=1$, by Theorem \ref{thm:Vfgammarsidon} we have that $V_{x^{q^s},\gamma}$ is an $r$-Sidon space in $\F_{q^{kt}}$ for any $t\geq r+1$ and by Theorem \ref{thm:monomialrspan} we also know that $\dim_{\fq}(V_{x^{q^s},\gamma}^r)=rk$. Moreover, we have also considered $V_{\mathrm{Tr}_{q^k/q},\gamma}$, showing that, except when $k\leq 3$, this is not a Sidon space. However, in Theorem \ref{thm:tracciadimensionerspan}, we have observed that $V_{\mathrm{Tr}_{q^k/q},\gamma}$ give us other examples of subspaces such that $\dim_{\fq}(V_{\mathrm{Tr}_{q^k/q},\gamma}^r)=rk$.\\
Now, we want to analyze the dimension of the $r$-span of $V_{f,\gamma}$ in the case in which $f_1(x)=x^{q^s}+\delta x^{q^{2s}}$ and $f_2(x)=x^{q^s}+\delta x^{q^{s(k-1)}}$.\\
Let start with $f_1(x)=x^{q^s}+\delta x^{q^{2s}}$.
We recall that $f_1(x)=x^{q^s}+\delta x^{q^{2s}}$ is scattered for $k=3,4$, whereas for $s=1, k\geq 5$ and $\delta\neq 0$ it is not scattered (see \cite[Remark 3.11]{zanella2019condition} and also \cite{montanucci2022class}). 
Moreover, if $n/k>2$ for any $\gamma\in\fqn$, $f_1(x)=x^{q^s}+\delta x^{q^{2s}}$ defines a Sidon space $V_{f_1,\gamma}\in\fqn$ (see Corollary 3.13 and Corollary 4.11 in \cite{CPSZSidon}). Let consider
\[
V_{f_1,\gamma}=\lbrace u+(u^{q^s}+\delta u^{q^{2s}})\gamma\colon u\in\fqk\rbrace.
\]
By MAGMA computations, we have computed $\dim_{\fq}(V_{f_1,\gamma}^r)$ by taking any $\delta\in\fqk,\delta\neq 0$, a random $\gamma$ primitive element of $\fqn$, and for the values of $(n,q,k,r,s)$ resumed in Table \ref{comptest1}. 
\begin{table}[htp]
\centering
\tabcolsep=3.0 mm
\begin{tabular}{|c|c|c|c|c|c|}
\hline
$r$ & $n$ & $q$ & $k$ &  $s$ & $\dim_{\fq}(V_{f_1,\gamma}^r)$ \\ \hline
3 & 25 & 2 & 5 &  1 & 20 \\ \hline 
3 & 30 & 2 & 6 &  1 & 24 \\ \hline 
3 & 35 & 2 & 7 &  1 & 28 \\ \hline 
3 & 40 & 2 & 8 &  1 & 32 \\ \hline 
4 & 24 & 2 & 4 &  1 & 20 \\ \hline 
4 & 35 & 2 & 5 &  1 & 25 \\ \hline 
4 & 36 & 2 & 6 &  1 & 30 \\ \hline 
4 & 42 & 2 & 7 &  1 & 35 \\ \hline 
4 & 48 & 2 & 8 &  1 & 40 \\ \hline 
5 & 28 & 2 & 4 &  1 & 24 \\ \hline 
5 & 35 & 2 & 5 &  1 & 30 \\ \hline
5 & 42 & 2 & 6 &  1 & 36 \\ \hline 
5 & 49 & 2 & 7 &  1 & 42 \\ \hline 
\end{tabular}
\caption{Dimensions of the $r$-span of $V_{f_1,\gamma}^r$ with $f_1(x)=x^{q^s}+\delta x^{q^{2s}}$}
\label{comptest1}
\end{table}
\\
In all the considered cases of Table \ref{comptest1} it results that $\dim_{\fq}(V_{f_1,\gamma}^r)=(r+1)k$. Hence, this seems to suggest that, once fixed $n,q,k,r,s$,  $\dim_{\fq}(V_{f_1,\gamma}^r)$ is $k$ times greater than $\dim_{\fq}(V_{x^{q^s},\gamma}^r)$.\\
\begin{table}[htp]
\centering
\tabcolsep=3.0 mm
\begin{tabular}{|c|c|c|c|c|c|}
\hline
$r$ & $n$ & $q$ & $k$ &  $s$ & $\dim_{\fq}(V_{f_1,\gamma}^r)$ \\ \hline
3 & 36 & 3 & 4 & 1 & 16 \\ \hline 
3 & 25 & 3 & 5 & 1 & 20 \\ \hline 
3 & 30 & 3 & 6 & 1 & 24 \\ \hline 
3 & 35 & 3 & 7 & 1 & 28 \\ \hline 
3 & 40 & 3 & 8 & 1 & 32 \\ \hline 
4 & 24 & 3 & 4 & 1 & 20 \\ \hline 
4 & 30 & 3 & 5 & 1 & 25 \\ \hline 
4 & 36 & 3 & 6 & 1 & 30 \\ \hline 
4 & 42 & 3 & 7 & 1 & 35 \\ \hline 
4 & 48 & 3 & 8 & 1 & 40 \\ \hline 
5 & 28 & 3 & 4 & 1 & 24 \\ \hline 
5 & 35 & 3 & 5 & 1 & 30 \\ \hline 
5 & 42 & 3 & 6 & 1 & 36 \\ \hline 
\end{tabular}
\caption{Dimensions of the $r$-span of $V_{f_2,\gamma}^r$ with $f_2(x)=x^{q^s}+\delta x^{q^{s(k-1)}}$}
\label{comptest2}
\end{table}
Now, let consider $f_2(x)=x^{q^s}+\delta x^{q^{s(k-1)}}$ and recall that this is scattered if and only if $\gcd(s,k)=1$ and $\mathrm{N}_{q^k/q}(\delta)\neq 1$ (see e.g. \cite{zanella2019condition}). Also, if $n/k>2$, for any $\gamma\in\fqn$, $f_2(x)=x^{q^s}+\delta x^{q^{s(k-1)}}$ defines a Sidon space $V_{f_2,\gamma}$ in $\fqn$, if and only if $\mathrm{N}_{q^k/q}(\delta)\neq 1$ or $k$ is odd (see Corollary 3.13 and Corollary 4.14 in \cite{CPSZSidon}).\\
Let consider
\[
V_{f_2,\gamma}=\lbrace u+(u^{q^s}+\delta u^{q^{s(k-1)}})\gamma\colon u\in\fqk\rbrace.
\]
By MAGMA computations, we have evaluated $\dim_{\fq}(V_{f_2,\gamma}^r)$ by taking any $\delta\in\fqk, \delta \neq 0$ and a random primitive element $\gamma$ of $\fqn$, and for the values of $(n,q,k,r,s)$ resumed in Table \ref{comptest2}. 
Note that when $k$ is even, to ensure the Sidon property on the subspace considered, we considered only those $\delta$ such that $\N_{q^k/q}(\delta)\ne 1$. As happens for $f_1(x)$, also for $f_2(x)$ it results that $\dim_{\fq}(V_{f_2,\gamma}^r)=(r+1)k$, for all the considered cases resumed in Table \ref{comptest2}.
\\
\\

\section{$B_r$-sets from $r$-Sidon spaces}\label{sec:Brsets}

As already mentioned in the Introduction, $B_r$-sets are of particular interest in number theory and have been studied since their introduction, see e.g. \cite{GenSS,newSidonsets,Sidonsurvey}.

In \cite{Roth}, the authors gave a way of constructing $B_r$-sets from $r$-Sidon spaces in $\fqn$ by looking at the elements of an $r$-Sidon space as the power of the same primitive element of $\fqn$.

We recall their result that generalizes the constructions given by Singer in \cite{singer} (see also \cite{Sidonsurvey}).

\begin{theorem}\cite[Theorem 48]{Roth}
\label{thm:Bose}
For an $r$-Sidon space $V \in \mathcal{G}(n, k)$, and denote by 
$\Tilde{V}$ the set of nonzero representatives of all one-dimensional subspaces of $V$. Let $\gamma \in \fqn$ be a primitive element of $\fqn$, then the set 
\[S= \{n_i : \gamma^{_i}\in\Tilde{V} \}\] 
is a $B_r$-set of size $\frac{q^k-1}{q-1}$ in $\mathbb{Z}_{(q^n-1)/(q-1)}$.
\end{theorem}

Therefore, by applying the above results to our constructions of Section \ref{sec:constr}, we obtain examples of $B_r$-sets.
\begin{corollary}
Let $n=kt$ for some $k,t \in \mathbb{N}$ and let $V=\lbrace u+f(u)\gamma\colon u\in\fqk\rbrace$ where $f(x)\in\mathcal{L}_{k,q}$ is scattered and $\gamma\in\fqn$ is a primitive element of $\fqn$. Let consider 
the set $\Tilde(V)$ of nonzero representatives of all one-dimensional subspaces of $V$. Then the set
\[
S=\lbrace n_i\colon \gamma^i\in\Tilde{V}\rbrace
\]
is $B_r$-sets of size $(q^{n/t}-1)/(q-1)$, for any $r \leq t-1$, in $\mathbb{Z}_{(q^n-1)/(q-1)}$.
\end{corollary}

\begin{example}
Let $q=3$ and let consider $V\in\mathcal{G}(16,4)$ defined by $f(x)=x^q+\delta x^{q^{k-1}}$ as follows
\[
V:=\lbrace u+\gamma (u^3+\delta u^{27})\colon u\in\F_{3^4}\rbrace
\]
where $\gamma\in\F_{3^{16}}$ is such that $\F_{3^{16}}=\F_{3^4}(\gamma)$ and $\delta\in\fqn\setminus\lbrace 0 \rbrace$. By Theorem \ref{thm:Vfgammarsidon} we know that $V$ is a $3$-Sidon space. By following Theorem \ref{thm:Bose} and by using MAGMA computations, we have constructed a set $T$, obtained by considering all the linear combinations over $\F_q$ of a basis of $V$. Then we constructed the following set $S$ by considering all the exponents of the elements of $T$ with respect to $\gamma$, by taking them modulo $(3^{16}-1)/2=21523360$. Moreover, in order to have $0\in S$ we have subtracted the minimum element of $S$ to any element of $S$, obtaining that, by Theorem \ref{thm:Bose}, the set
\begin{equation}
\begin{aligned}
    S=&\lbrace 0, 264883, 810267, 2455111, 2462280, 3170728, 4135313, 4565718, 5110363, 6191430,\\
    &6840447, 8527892, 8892224, 10182757, 10447118, 10587200, 10636907, 10855603,\\
    &11398308, 11752618, 12064245, 12292344, 12368319, 12443431, 12473660, 13377589,\\
    &14811094, 14890807, 15040932, 15465744, 15806165, 16397984,16566024, 17550439, \\
    &18128170, 18144682, 18904208, 18967185, 20329901, 20791977 \rbrace 
\end{aligned}
\end{equation}
is a $B_3$-set of size $(3^4-1)/2=40$ in $\mathbb{Z}_{(3^{16}-1)/2}=\mathbb{Z}_{21523360}$.
\end{example}

\section*{Acknowledgements}

The research was supported by the project ``COMBINE'' of the University of Campania ``Luigi Vanvitelli'' and was partially supported by the Italian National Group for Algebraic and Geometric Structures and their Applications (GNSAGA - INdAM).

\end{document}